\documentclass[leqno,11pt]{amsart}

\usepackage{geometry}

\usepackage[all,cmtip]{xy}
\usepackage{tikz-cd}

\usepackage[all]{xy} 

\usepackage{comment}

\usepackage{mathtools}
\usepackage{amsmath, amssymb, amsfonts, latexsym, mdwlist, amsthm}
\usepackage{subfig}
\usepackage{graphicx}
\usepackage{wrapfig}

\usepackage[bookmarks, colorlinks, breaklinks, pdftitle={},
pdfauthor={}]{hyperref}
\hypersetup{linkcolor=blue,citecolor=blue,filecolor=black,urlcolor=blue}

\usepackage{tikz}
\usetikzlibrary{calc,trees,positioning,arrows,chains,shapes.geometric,%
    decorations.pathreplacing,decorations.pathmorphing,shapes,%
    matrix,shapes.symbols}

\tikzset{
>=stealth',
  punktchain/.style={
    rectangle,
    rounded corners,
    draw=black, thick,
    minimum height=3em,
    text centered,
    on chain},
  line/.style={draw, thick, <-},
  element/.style={
    tape,
    top color=white,
    bottom color=blue!50!black!60!,
    minimum width=8em,
    draw=blue!40!black!90, very thick,
    text width=10em,
    minimum height=3.5em,
    text centered,
    on chain},
  every join/.style={->, thick,shorten >=1pt},
  decoration={brace},
  tuborg/.style={decorate},
  tubnode/.style={midway, right=2pt},
}

\usepackage{paralist}
\setdefaultenum{(a)}{(i)}{}{}
\usepackage{enumitem} 
\usepackage{graphicx}

\usetikzlibrary{patterns}

\renewcommand\_{^{}_}
\renewcommand\;{\hspace{.6pt}}

\newcommand\C{\mathbb C}

\newcommand\R{\mathbb R}

\newcommand\Z{\mathbb Z}

\newcommand\cA{\mathcal A}
\newcommand\cW{\mathcal W}
\newcommand\cD{\mathcal D}
\newcommand\cH{\mathcal H}
\newcommand\cO{\mathcal O}

\newcommand\ch{\operatorname{ch}}

\newcommand\Hom{\operatorname{Hom}}

\newcommand\Pic{\operatorname{Pic}}

\newcommand\Coh{\operatorname{Coh}}
\newcommand\cok{\operatorname{coker}}
\newcommand\bn{\operatorname{BN}}
\newcommand\ev{\operatorname{ev}}
\newcommand\p{\operatorname{P}}
\renewcommand\v{\mathsf v}

\newcommand\im{\operatorname{im}}

\newcommand\al{\alpha}
\def\Stab{\mathop{\mathrm{Stab}}\nolimits}

\def\abs#1{\left\lvert#1\right\rvert}

\newcommand\beq[1]{\begin{equation}\label{#1}}
\newcommand\eeq{\end{equation}}
\newcommand\beqa{\begin{eqnarray*}}
\newcommand\eeqa{\end{eqnarray*}}

\makeatletter
\newtheorem*{rep@theorem}{\rep@title}
\newcommand{\newreptheorem}[2]{%
\newenvironment{rep#1}[1]{%
 \def\rep@title{#2 \ref{##1}}%
 \begin{rep@theorem}}%
 {\end{rep@theorem}}}
\makeatother

\newtheorem{Thm}{Theorem}[section]
\newreptheorem{Thm}{Theorem}
\newtheorem{Thm*}{Theorem}
\newtheorem{Prop}[Thm]{Proposition}

\newtheorem{Lem}[Thm]{Lemma}

\newtheorem{Cor}[Thm]{Corollary}
\newreptheorem{Cor}{Corollary}

\newtheorem{thm-int}{Theorem}

\theoremstyle{definition}
\newtheorem{Def-s}[Thm]{Definition}
\newtheorem{Def}[Thm]{Definition}
\newtheorem{Rem}[Thm]{Remark}

\newcommand{\ignore}[1]{}

\begin{document}
\title{\ \vspace{-15mm} \\ Hyperk\"{a}hler varieties as Brill-Noether loci on curves \vspace{-3mm}}
\author{Soheyla Feyzbakhsh\vspace{-6mm}}
\maketitle

\begin{abstract}
	Consider the moduli space $M_C(r; K_C)$ of stable rank r vector bundles on a curve $C$ with canonical determinant, and let $h$ be the maximum number of linearly independent global sections of these bundles. If $C$ embeds in a K3 surface $X$ as a generator of $\Pic(X)$ and the genus of $C$ is sufficiently high, we show the Brill-Noether locus $\bn_C \subset M_C(r; K_C)$ of bundles with $h$ global sections is a smooth projective Hyperk\"{a}hler manifold of dimension $2g -2r \lfloor \frac{g}{r}\rfloor$, isomorphic to a moduli space of stable vector bundles on $X$. 
\end{abstract}
\bigskip

\section{Introduction}\label{section.intro}


In this paper, we show that specific Brill-Noether loci of rank $r$-stable vector bundles on K3 curves are isomorphic to moduli spaces of stable vector bundles on K3 surfaces. This generalises Mukai's program \cite{mukai:non-abelian-brill-noether,arbarello:maukai-program,feyz:mukai-program,feyz:mukai-program-ii} to Brill-Noether loci on K3 curves of dimension higher than $2$.

 Fix $(r, k) \in \mathbb{Z}_{>0} \oplus \mathbb{Z}_{>0}$ such that $\gcd(r, k) =1$ and $k < r$. Then take $g \gg 0$. Let $(X, H)$ be a polarised K3 surface with $\Pic(X) = \mathbb{Z} \cdot H$ and let $C$ be any curve in the linear system $|H|$ of genus $g$. There is a unique $s \in \mathbb{Z}$ such that
\begin{equation}\label{assumption}
-2 \leq k^2H^2 -2rs < -2 +2r\, . 
\end{equation} 
We further assume $\gcd(s, k) =1$. Consider the Brill-Noether locus $\bn_C(r, k(2g-2), r+s)$ of semistable\footnote{See Definition \ref{def-stability}.
} rank $r$-vector bundles on $C$ having degree $kH^2 =k(2g-2)$ and possessing at least $r+s$ linearly independent global sections. Also $M_X(\v)$ denotes the moduli space of $H$-Gieseker semistable sheaves on $X$ with Mukai vector $\v = (r, kH, s)$. It is a (non-empty) smooth quasi-projective variety\footnote{See for instance \cite[Chapter 10, Corollary 2.1 \& Theorem 2.7]{huybretchts:lectures-on-k3-surfaces}.} of dimension $\v^2+2 = k^2H^2-2rs+2$.
\begin{Thm}\label{thm}
	There is an isomorphism 
	\begin{equation*}
	\Psi \colon M_X(\v) \rightarrow \bn_C(r, k(2g-2), r+s)
	\end{equation*}
	which sends a vector bundle\footnote{Any $H$-Gieseker semistable sheaf of Mukai vector $\v$ is locally free, see Lemma \ref{lem-locally-free}.} $E$ on $X$ to its restriction $E|_C$.    
\end{Thm}
The above Theorem says that for any vector bundle $F$ on the curve $C$ in the Brill-Noether locus $\bn_C(r, k(2g-2), r+s)$, there exists a unique vector bundle on the surface $X$ whose restriction to $C$ is isomorphic to $F$. In particular, we obtain the following: 
\begin{Cor}
	Any vector bundle $F \in \bn_C(r, k(2g-2), r+s)$ is stable (cannot be strictly semistable) and $\wedge^r F = \omega_C^{\otimes k}$ with $h^0(C, F) = r+s$. 
\end{Cor}


\subsection*{Proof ideas} To prove Theorem \ref{thm}, we go through the following steps:  
\begin{enumerate*}
	\item Consider the embedding $\iota \colon C \hookrightarrow X$. The push-forward $\iota_* F$ for any vector bundle $F \in \bn_C(r, kH^2, r+s)$ is semistable in the large volume limit of the space of Bridgeland stability conditions on $X$. We move down and study walls for objects of class $\ch(\iota_*F)$ between the large volume limit and the Brill-Noether wall that the structure sheaf $\cO_X$ is making. The first wall is made by stable sheaves with Mukai vector $\v= (r, kH, s)$. We show $\iota_* F$ gets destabilised along this wall if and only if $F = E|_C$ for a sheaf $E \in M_X(\v)$.
	
	\item The assumption \eqref{assumption} implies that any stable coherent sheaf of Mukai vector $\v$ is locally-free. We show that there is no wall for objects with Mukai vector $\v$ between the large volume limit and the Brill-Noether wall. That's why we get $h^0(X, E) = r+s$ for any $E \in M_X(\v)$. Hence, there is a well-defined map   
	\begin{align*}
	\Psi \colon M_X(\v)  &\rightarrow \bn_C(r, kH^2, r+s)\\
	E & \mapsto E|_C. 
	\end{align*}
	Then a usual wall-crossing argument gives injectivity of $\Psi$ due to the uniqueness of Harder-Narasimhan filtration. 

	\item To prove $\Psi$ is surjective we apply the technique developed in \cite{feyz:mukai-program} which bounds the number of global sections of sheaves on $X$ in terms of the length of a convex polygon (which is the Harder-Narasimhan polygon in the Brill-Noether region). This implies that if $h^0(X, \iota_*F) \geq r+s$, then $F$ is the restriction of a vector bundle $E \in M_X(\v)$ to the curve $C$, so in particular, $\Psi$ is surjective.  	
	
	\item Finally, any vector bundle $E \in M_X(\v)$ is generated by global sections and the kernel 
	\begin{equation*}
	K_E \coloneqq \ker (\cO_X^{\oplus r+s} \xrightarrow{\text{ev}} E). 
	\end{equation*}
	is a slope-stable vector bundle. By applying a wall-crossing argument, we show that Hom$(K_E, E(-H)[1]) = 0$. Then a usual deformation theory argument implies that the derivative $d \Psi$ is surjetive, and so $\Psi$ is an isomorphism.  
	 
    

\end{enumerate*}

Note that in Theorem \ref{thm}, the assumption $g \gg 0$ is necessary for all the above steps. Remark \ref{Rem-big enough} gives a list of inequalities that $g$ must satisfy to get Theorem \ref{thm}.


\subsection*{Outlook}
In this paper, we only work on curves on K3 surfaces of Picard rank one. More generally one can consider a polarised K3 surfaces $(X, H)$ of arbitrary Picard rank and pick a curve $C \in |H|$ of genus $g(C) \gg 0$. For any $r, k \in \mathbb{Z}^{>0}$, we define\footnote{One can replace the structure sheaf $\cO_C$ with any other vector bundle on $C$ and consider the corresponding Brill-Noether locus.} 
\begin{equation*}
h_{r, k} \coloneqq \max\left\{h^0(C, F) \colon \text{rank $r$-stable vector bundles $F$ on $C$ with $\wedge^rF = \omega_C^{\otimes k}$
}  
\right\}.
\end{equation*}
Consider the Brill-Noether locus $\bn_C^{\text{st}}(r,\,  \omega_C^{\otimes k} ,\, h_{r, k})$ of stable rank $r$-vector bundles $F$ on $C$ of fixed determinant $ \omega_C^{\otimes k}$ and having $h_{r, k}$ linearly independent global sections. Then the natural question is whether this Brill-Noether locus is a Hyperk\"{a}hler variety when $g(C)$ is high enough. I expect to answer this question in the future by applying a more general wall-crossing argument. 


Note that the assumption $g(C) \gg 0$ is necessary as there are examples of the above Brill-Noether loci on K3 curves of genus $7$ \cite{mukai:non-abelian-brill-noether} and genus $12$ \cite{feyz-genus-12} which are smooth Fano varieties. A similar technique as in this paper can be applied to generalise Theorem \ref{thm} to 
\begin{enumerate}
	\item polarised K3 surfaces $(X, H)$ such that $H^2$ divides $H.D$ for all curve classes $D$ on $X$, and 
	\item curves $C \hookrightarrow X$ which are not necessarily in the linear system $|H|$.  
\end{enumerate}

\subsection*{Plan} In Section \ref{section.review}, we review Bridgeland stability conditions on the bounded derived category of coherent sheaves on a K3 surface $X$. Section \ref{section.wall-crossing} analyses walls for the push-forward of vector bundles in our Brill-Noether locus to the K3 surface $X$ and as a result, proves Theorem \ref{thm}. The computations for the location of walls are all postponed to Section \ref{section.location}.   

\subsection*{Acknowledgement}
I would like to thank Daniel Huybrechts who was the first person to encourage me to think about higher-dimensional Hyperk\"{a}hler varieties as Brill-Noether loci on K3 curves. I would also like to thank Enrico Arbarello, Arend Bayer, Ignacio Barros Reyes, Chunyi Li, Emanuele Macr\`{i}, Lenny Taelman, Richard Thomas, and Yukinobu Toda for helpful discussions and comments about this work. The author was supported by Marie Sk\l odowska--Curie individual fellowships 887438 and EPSRC postdoctoral fellowship EP/T018658/1.

\section{Review: Bridgeland stability conditions on K3 surfaces}\label{section.review}
In this section, we review the description of a slice of the \textit{space  of stability conditions} on the bounded derived category of coherent sheaves on a K3 surface given in \cite[Section 1-7]{bridgeland:K3-surfaces}.

Let $X$ be a projective scheme over $\mathbb{C}$ of dimension $\dim(X) \geq 1$, and let $H$ be an ample line bundle on $X$. The Hilbert polynomial $P(E, m)$ is given by $m \mapsto \chi(\cO_X, E \otimes \cO_X(mH))$. It can be uniquely written in the form
\begin{equation*}
P(E, m) = \sum_{i=0}^{\dim(E)}\al_i(E)\frac{m^i}{i!}
\end{equation*}
with integral coefficients $\al_i(E)$. The reduced Hilbert polynomial $p(E, m)$ of a coherent sheaf $E$ of dimension $d$ is defined by
\begin{equation*}
p(E, m) = \frac{P(E, m)}{\al_d(E)}.
\end{equation*}
\begin{Def}{\cite[Definition 1.2.4]{huybrechts:geometry-of-moduli-space-of-sheaves}}\label{def-stability}
	A coherent sheaf $E$ of dimension $d$ is (semi)stable if $E$ is pure and for any proper subsheaf $F \subset E$ one has $p(F, m) \,(\leq)\, p(E, m)$ for $m \gg 0$. Here $(\le)$ denotes $<$ for stability and $\le$ for semistability.  
\end{Def}

From now on, we always assume $(X,H)$ is a smooth polarized K3 surface over $\mathbb{C}$ with Pic$(X) = \Z \cdot H$. Denote the bounded derived category of coherent sheaves on $X$ by $\cD(X)$ and its Grothendieck group by $K(X):=K(\cD(X))$. Given an object $E \in \cD(X)$, we write $\ch(E) = (\ch_0(E), \ch_1(E), \ch_2(E)) \in H^*(X, \Z)$ for its Chern characters. The Mukai vector of an object $E \in \cD(X)$ is an element of the numerical Grothendieck group $\mathcal{N}(X) = \mathbb{Z} \oplus \text{NS}(X) \oplus \mathbb{Z} \cong \mathbb{Z}^3$ defined via
\begin{equation*}
v(E) \coloneqq \big(\ch_0(E),\, \ch_1(E),\, \ch_0(E) +\ch_2(E)\big) = \text{ch}(E)\sqrt{\text{td}(X)}\,. 
\end{equation*}
The Mukai bilinear form
\begin{equation*}
\left \langle v(E), v(E')\right \rangle = \ch_1(E)\ch_1(E') + \ch_0(E)\left(\ch_0(E') +\ch_2(E')\right) + \ch_0(E')\left(\ch_0(E) +\ch_2(E)\right)  
\end{equation*}
makes $\mathcal{N}(X)$ into a lattice of signature $(2,1)$. The Riemann-Roch theorem implies that this form is the negative of the Euler form, defined as
\begin{equation*}
\chi(E,E') = \sum_{i} (-1)^{i} \dim_{\mathbb{C}} \text{Hom}_X^{i}(E,E')  =  -\left \langle v(E), v(E')\right \rangle.
\end{equation*}


The slope of a coherent sheaf $E \in \Coh(X)$ is defined by
\[ \mu\_{H}(E) := \begin{cases}
\frac{H.\ch_1(E)}{H^2\ch_0(E)} & \text{if $\ch_0(E) > 0$} \\
+\infty & \text{if $\ch_0(E) = 0$}.
\end{cases}\]
This leads to the usual notion of $\mu_H$-stability. Associated to it every sheaf $E$ has a Harder-Narasimhan filtration. Its graded pieces have slopes whose maximum we denote by $\mu_H^+(E)$ and minimum by $\mu_H^-(E)$.

For any $b \in \mathbb{R}$, let $\cA(b)\subset\cD(X)$ denote the abelian category of complexes
\begin{equation}\label{Abdef}
\mathcal{A}(b)\ =\ \big\{E^{-1} \xrightarrow{\,d\,} E^0 \ \colon\ \mu_H^{+}(\ker d) \leq b \,,\  \mu_H^{-}(\cok d) > b \big\}.
\end{equation}
Then $\cA(b)$ is the heart of a bounded t-structure on $\cD(X)$ by \cite[Lemma 6.1]{bridgeland:K3-surfaces}. For any pair $(b,w) \in \R^2$, we define the group homomorphism $Z_{b,w} \colon K(X) \to \C$ by
\begin{equation} \label{eq:Zab}
Z_{b,w}(E) := -\ch_2(E) + w \ch_0(E)  + i \bigg(\frac{H\ch_1(E)}{H^2} -b \ch_0(E) \bigg).
\end{equation}
Define the function $\Gamma \colon \mathbb{R} \rightarrow \mathbb{R}$ as 
	\[ \Gamma(b) \coloneqq  \begin{cases}
 b^2 \left(\frac{H^2}{2}+1\right)-1 & \text{if $b\neq 0$} \\
0 & \text{if $b = 0$}.
\end{cases}\]
By abuse of notations, we also denote the graph of $\Gamma$ by \textit{curve $\Gamma$} (see Figure \ref{projetcion-fig}). We define
\begin{equation*}
U \coloneqq \left\{(b,w) \in \mathbb{R}^2 \colon w > \Gamma(b)  \right\}.
\end{equation*}
In figures, we will plot the $(b,w)$-plane simultaneously with the image of the projection map 
	\begin{eqnarray*}
		\Pi\colon\ K(X) \setminus \big\{E \colon \ch_0(E) = 0\big\}\! &\longrightarrow& \R^2, \\
		E &\ensuremath{\shortmid\joinrel\relbar\joinrel\rightarrow}& \!\!\bigg(\frac{\ch_1(E).H}{\ch_0(E)H^2}\,,\, \frac{\ch_2(E)}{\ch_0(E)}\bigg).
	\end{eqnarray*}
	
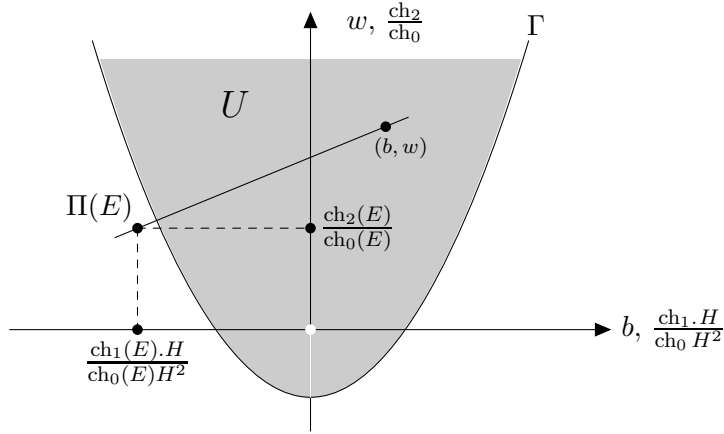
\begin{figure}[h]
	\begin{centering}
		\definecolor{zzttqq}{rgb}{0.27,0.27,0.27}
		\definecolor{qqqqff}{rgb}{0.33,0.33,0.33}
		\definecolor{uququq}{rgb}{0.25,0.25,0.25}
		\definecolor{xdxdff}{rgb}{0.66,0.66,0.66}
		
		\begin{tikzpicture}[line cap=round,line join=round,>=triangle 45,x=1.0cm,y=0.9cm]
		
		\draw  (4, 0) node [right ] {$b,\,\frac{\ch_1\!.\;H}{\ch_0H^2}$};


		\fill [fill=gray!40!white] (0,-1) parabola (2.8,4) parabola [bend at end] (-2.8,4) parabola [bend at end] (0,-1);
		
		\draw[->,color=black] (-4,0) -- (4,0);
		\draw  (0,-1) parabola (2.9,4.27); 
		\draw  (0,-1) parabola (-2.9,4.27); 
		\draw  (3 , 4.2) node [above] {$\Gamma$};

		\draw[->,color=black] (0,-1.5) -- (0,4.7);
		\draw  (1, 4.1) node [above ] {$w,\,\frac{\ch_2}{\ch_0}$};
		
		\draw [dashed, color=black] (-2.3,1.5) -- (-2.3,0);
		\draw [dashed, color=black] (-2.3, 1.5) -- (0, 1.5);
		\draw [color=black] (-2.6, 1.36) -- (1.3, 3.14);
		\draw [color=white] (0, 0) -- (0, -1);
		
		\draw  (-2.8, 1.8) node {$\Pi(E)$};
		\draw  (-1, 3) node [above] {\Large{$U$}};
		\draw  (0, 1.5) node [right] {$\frac{\ch_2(E)}{\ch_0(E)}$};
		\draw  (-2.3 , 0) node [below] {$\frac{\ch_1(E).H}{\ch_0(E)H^2}$};
		\begin{scriptsize}
		\fill (0, 1.5) circle (2pt);
		\fill[color=white] (0, 0) circle (2pt);
		\fill (-2.3,0) circle (2pt);
		\fill (-2.3,1.5) circle (2pt);
		\fill (1,3) circle (2pt);
		\draw  (1.2, 2.96) node [below] {$(b,w)$};
		
		\end{scriptsize}
		
		\end{tikzpicture}
		
		\caption{$(b,w)$-plane and the projection $\Pi(E)$}
		
		\label{projetcion-fig}
		
	\end{centering}
\end{figure}

Consider the slope function
\[ \nu\_{{b,w}}\colon \cA(b) \to \R \cup \{+\infty\}, \quad \nu\_{b,w} (E) \coloneqq \begin{cases}
-\frac{\text{Re}[Z_{b,w}(E)]}{\text{Im}[Z_{b,w}(E)]} & \text{if $\text{Im}[Z_{b,w}(E)] > 0$} \\
+\infty & \text{if $\text{Im}[Z_{b,w}(E)] = 0$.}
\end{cases}
\]
This defines our notion of stability in $\cA(b)$:
\begin{Def}
	Fix $w> \Gamma(b)$. We say $E\in\cD(X)$ is (semi)stable with respect to the pair $\sigma_{b,w} \coloneqq \left(\cA(b),\, Z_{b,w} \right)$ if and only if
	\begin{itemize}
		\item $E[k]\in\cA(b)$ for some $k\in\Z$, and
		\item $\nu\_{b,w}(F)\,(\le)\,\nu\_{b,w}\big(E[k]/F\big)$ for all non-trivial subobjects $F\hookrightarrow E[k]$ in $\cA(b)$.
	\end{itemize}
\end{Def}
Let $E$ be a semistable coherent sheaf, or more generally a $\sigma_{b,w}$-semistable object for some $(b,w) \in U$. 
By \cite[Remark 2.3(a)]{feyz-li:clifford-indices}\footnote{Note that the curve $\Gamma$ defined in this paper lies above the curve $\Gamma$ considered in \cite{feyz-li:clifford-indices}.}, the projection $\Pi(E)$ does not lie in $U$. Thus the same argument as in \cite[Lemma 6.2]{bridgeland:K3-surfaces} shows that the pair $\sigma_{b,w} = \left(\cA(b),\, Z_{b,w} \right)$ is a Bridgeland stability condition on $\cD(X)$, see \cite[Definition 1.1 \& Proposition 5.3]{bridgeland:stability-condition-on-triangulated-category}\footnote{Up to the action of $\widetilde{\text{GL}}^{+}(2;\mathbb{R})$, the stability conditions $\sigma_{b,w}$ are the same as the stability conditions defined in \cite[section 6]{bridgeland:K3-surfaces}.}. This in particular implies that every object $E \in \cA(b)$ admits a \textit{Harder--Narasimhan filtration} which is a finite sequence of objects in $\cA(b)$:
\[0=E_0\subset E_1\subset E_2\subset\dots\subset E_k=E\]
whose factors $E_i/E_{i-1}$ are $\sigma_{b,w}$-semistable and $\nu_{b,w}(E_1)>\nu_{b,w}(E_2)>\dots>\nu_{b,w}(E/E_{k-1})$. We denote $\nu_{b,w}^+(E)\coloneqq \nu_{b,w}(E_1)$ and $\nu_{b,w}^-(E) \coloneqq  \nu_{b,w}(E_k)$. Summarizing, we have the following: 
\begin{Thm}[{\cite[Section 1]{bridgeland:K3-surfaces}}] \label{thm:stabconstr}
	For any pair $(b,w) \in U$, the pair $\sigma_{b,w} = \left(\cA(b), Z_{b,w}\right)$
	defines a stability condition on $\cD(X)$. Moreover, the map from $U$ to the space of stability conditions $\Stab(X)$ on $\cD(X)$ given by $(b,w) \mapsto \sigma_{b,w}$ is continuous.
\end{Thm}
The second part of Theorem \ref{thm:stabconstr} implies that the two-dimensional family of stability conditions $\sigma_{b,w}$ satisfies wall-crossing as $b$ and $w$ vary, see for instance \cite[Proposition 4.1]{feyz:thomas-noether-loci} or \cite[Proposition 2.4]{feyz-li:clifford-indices}.
\begin{Prop}[\textbf{Wall and chamber structure}]\label{locally finite set of walls}
	Fix a non-zero object $E \in \mathcal{D}(X)$. There exists a collection of line segments (walls) $\cW_E^i$ in $U$ (called ``\emph{walls}") which are locally finite and satisfy 
	\begin{itemize*}
		\item[\emph{(a)}] The extension of each line segment $\cW_E^i$ passes through the point $\Pi(E)$ if $\ch_0(E) \neq 0$, or has fixed slope $\frac{\ch_2(E)H^2}{\ch_1(E).H}$ if $\ch_0(E) = 0$. 
		
		\item[\emph{(b)}] An endpoint of the segment $\cW_E^i$ is either on the curve $\Gamma$ or on the line segment through $(0, 0)$ to $(0,-1)$.
		
		\item[\emph{(c)}] The $\sigma\_{b,w}$-(semi)stability of $E$ is unchanged as $(b,w)$ varies within any connected component (called a ``\emph{chamber}") of $U \setminus \bigcup_{i \in I} \cW_E^i $.
		
		\item[\emph{(d)}] For any wall $\cW_E^i$ there is $k_i \in \mathbb{Z}$ and a map $f\colon F\to E[k_i]$ in $\cD(X)$ such that
		\begin{itemize}
			\item for any $(b,w) \in \cW_E^i$, the objects $E[k_i],\,F$ lies in the heart $\cA(b)$,
			\item $E[k_i]$ is $\nu\_{b,w}$-semistable with $\nu\_{b,w}(E)=\nu\_{b,w}(F)=\,\mathrm{slope}\,(\cW_E^i)$ constant on $\cW_E^i$, and
			\item $f$ is an injection $F\subset E[k_i]$ in $\cA(b)$ which strictly destabilises $E[k_i]$ for $(b,w)$ in one of the two chambers adjacent to the wall $\cW_E^i$.
			
		\end{itemize} 
	\end{itemize*} 
	\begin{figure} [h]
		\begin{centering}
			
			\begin{tikzpicture}[line cap=round,line join=round,>=triangle 45,x=1.0cm,y=1.0cm]

			\fill [fill=gray!30!white] (-0.5,-.5) parabola (1.47, 3.03) parabola [bend at end] (-2.47,3.03) parabola [bend at end] (-0.5,-.5);
			
			\fill [fill=gray!30!white] (-8,-.5) parabola (-6.03, 3.03) parabola [bend at end] (-9.97,3.03) parabola [bend at end] (-8,-.5);

			\draw[->,color=black] (-10.5,0) -- (-5.5,0);
			\draw[->,color=black] (-3,0) -- (2,0);
			\draw[->,color=black] (-8,-1) -- (-8,3.5);
			\draw[->,color=black] (-0.5,-1) -- (-0.5,3.5);

			\draw [] (-0.5,-.5) parabola (1.5,3.12); 
			\draw [] (-0.5,-.5) parabola (-2.5,3.12); 
			\draw [] (-8,-.5) parabola (-10,3.12); 
			\draw [] (-8,-.5) parabola (-6,3.12);

			\draw[color=black, dashed] (-10.5,2.8) -- (-6,1);
			\draw[color=black, dashed] (-10.5,1.8) -- (-6.5,0.2);

			\draw[color=black,semithick] (-9.8,2.52) -- (-6.6,1.24);
			\draw[color=black,semithick] (-9.45,1.38) -- (-7.,.4);
			
			\draw (-10.5,1.8) node [left] {$\cW_E^2$};
			\draw (-10.5,2.8) node [left] {$\cW_E^1$};

			\draw (.8,3.5) node [right] {\large{$\ch_0(E) \neq 0$}};
			\draw (-6.8,3.5) node [right] {\large{$\ch_0(E) = 0$}};
			\draw (-5.5,0) node [below] {$b, \frac{\ch_1.H}{H^2\ch_0}$};
			\draw (-8,3.5) node [above] {$w, \frac{\ch_2}{\ch_0}$};
			
			\draw (2,0) node [below] {$b, \frac{\ch_1.H}{H^2\ch_0}$};
			\draw (-0.5,3.5) node [above] {$w, \frac{\ch_2}{\ch_0}$};
			
			\draw (-7.3,2.5) node [right] {\Large{$U$}};
			\draw (0.2,2.5) node [right] {\Large{$U$}};

			\draw (1.5, 1.2) node [right] {$\Pi(E)$};
			
			\draw[color=black, dashed] (1.5, 1.2) -- (-2.9,2.8);
			\draw[color=black, dashed] (1.5, 1.2) -- (-2.2, .7);
			
			\draw (-2.9,2.8) node [left] {$\cW_E^1$};
			\draw (-2.2, .7) node [left] {$\cW_E^2$};

			\draw[color=black, semithick] (-2.3 ,2.58) -- (.88,1.423);
			
			\draw[color=black, semithick] (-1.7,.764) -- (0.8,1.108);
			
			\draw[color=white, thick] (-.5, -.5) -- (-.5, 0);
				\draw[color=white, thick] (-8, -.5) -- (-8, 0);
			
			\begin{scriptsize}
			
			%
			%
			%

			\fill [color=black] (1.5,1.2) circle (2pt);
			
					\fill [color=white] (-0.5,0) circle (2pt);
					\fill [color=white] (-8,0) circle (2pt);
			%

			
			%
			%
			%
			%
			%
			%
			

			\end{scriptsize}
			
			\end{tikzpicture}
			
			\caption{Walls $\cW_E^i$ for $E$.}
			
			\label{wall.figure}
			
		\end{centering}
		
	\end{figure}
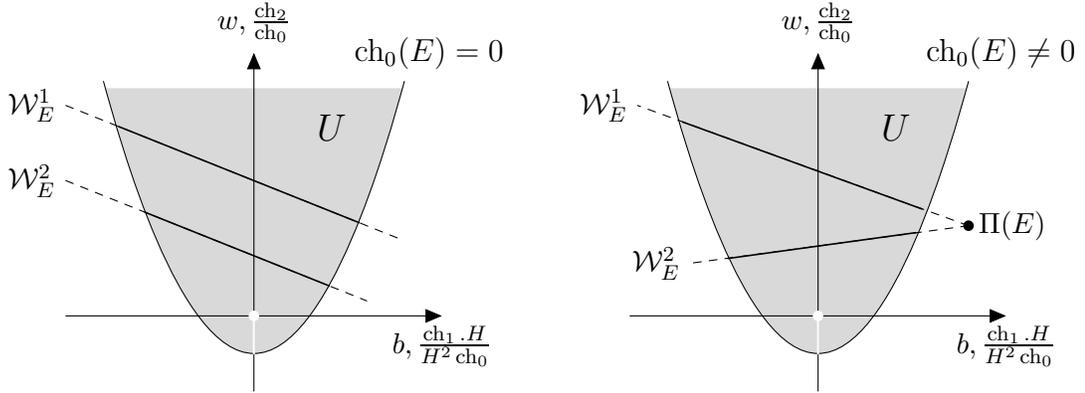
\end{Prop}

\section{Wall-crossing}\label{section.wall-crossing}
Let $(X,H)$ be a smooth polarized K3 surface over $\mathbb{C}$ with Pic$(X) = \Z \cdot H$, and let $C \in |H|$ be any curve of genus $g = \frac{1}{2}H^2 +1$. 
The Brill-Noether loci of vector bundles on $C$ when genus $g$ is low have been highly studied, see e.g. \cite{mukai-curves-k3surface-genus-11,mukai:curve-k3surface-genus-less-than-10,mukai:new-development-in-the-theory-of-fano-threefold,mukai:non-abelian-brill-noether,mukai:vector-bundles-and-brill-noether}. But in this section, we analyse these loci when genus $g$ is high.

As in the Introduction, fix $(r, k) \in \mathbb{Z}_{>0} \oplus \mathbb{Z}_{>0}$ such that $\gcd(r, k) =1$ and $k < r$. Then we take $g \gg 0$ and consider polarised K3 surfaces $(X, H)$ of genus $g = \frac{1}{2}H^2+1$. Consider the Mukai vector $\v \coloneqq (r, kH, s)$ such that $\gcd(s, k) =1$ and the assumption \eqref{assumption} holds:
\begin{equation*}
-2 \leq \v^2= k^2H^2 -2rs < -2 +2r\,.
\end{equation*}  
	This implies
\begin{equation}\label{bound for s}
-1+\frac{1}{r} +\frac{k^2}{2r}H^2< \  s \ \leq \frac{1}{r}+\frac{k^2}{2r}H^2\, . 
\end{equation}
We set Mukai vector $\al \coloneqq (s, -kH, r)$, so 
\begin{equation*}
\Pi(\v) = \left(\frac{k}{r},\, \frac{s-r}{r} \right)\ ,  \qquad 	\Pi(\al) = 
\left(-\frac{k}{s}\ ,\ \frac{r-s}{s}   \right),
\end{equation*}
and
	\begin{equation*}
	\Pi(\v(-H)) = \Pi\left(r, (k-r)H, s-kH^2+ \frac{r}{2} H^2 \right) = \left(\frac{k-r}{r} \ ,\ \frac{s-r}{r} -\frac{k}{r}H^2+ \frac{1}{2}H^2  \right). 
	\end{equation*}
	
\begin{Def}\label{def-ellstar}
	Let $\ell^*$ be the lowest line of slope $H^2(\frac{k}{r} -\frac{1}{2})$ which intersects the curve $\Gamma$ at two points with $b$-values $b_1^*<b_2^*$ satisfying
	\begin{equation}\label{cond.for ellstar}
	\max \left\{b_1^*- \frac{k-r}{r} \, , \  \frac{k}{r} - b_2^*  \right\} \,\leq\, 
	\frac{1}{r^2(r+1)} 
	\end{equation}
\end{Def}
\begin{figure}[h]
	\begin{centering}
		\definecolor{zzttqq}{rgb}{0.27,0.27,0.27}
		\definecolor{qqqqff}{rgb}{0.33,0.33,0.33}
		\definecolor{uququq}{rgb}{0.25,0.25,0.25}
		\definecolor{xdxdff}{rgb}{0.66,0.66,0.66}
		
		\begin{tikzpicture}[line cap=round,line join=round,>=triangle 45,x=1.0cm,y=1.0cm]
		
		\draw[->,color=black] (-8,0) -- (7.5,0);
		\draw  (7.5, 0) node [right ] {$b,\,\frac{\ch_1.H}{\ch_0H^2}$};


		
		\draw  (0,-1.5) [color=blue, thick]parabola (6,7.5); 
		\draw  (0,-1.5) [color=blue, thick] parabola (-6,7.5); 
		\draw (5.55,7.5) node [above] {$\Gamma$};
		\draw [color=blue, thick] (0, 0)--(0, -1.5);
		
		\draw [thick] (7,3.6)--(-8,7.1);
		\draw [thick] (6,1.2)--(-6, 4.2);
		\draw [thick] (-3.5, -2.6) -- (7, 5.1);
		
		\draw [thick] (0, 0) -- (-6.8, 7.5);
		
		\draw [thick] (-1, -2.5) -- (-6., 6.6);

		\draw[->,color=black] (0,-3) -- (0,8.5);
		\draw  (0, 8.5) node [above ] {$w,\,\frac{\ch_2}{\ch_0}$};

		\draw  (4.3, 3.) node [right] {$b_2^{\v}$};
		\draw  (4.6, 4.1) node [above] {$\tilde{b}_2$};
		\draw  (7, 5.1) node [right] {\Large{$\ell_{\v}$}}; 
		\draw  (6,1.2) node [right] {\Large{$\ell^*$}};
		\draw  (7,3.6) node [right] {\Large{$\tilde{\ell}$}};
		
		\draw  (3.74, 1.8) node [below] {$b_2^*$};
		
		\draw  (-1.7, -1.28) node [left] {$\Pi(\al)$};
		\draw  (-1, -2.5) node [below] {\Large{$\ell_{\al}$}};
		\draw  (5.45,3.95) node [above] {$\Pi(\v)$};
		
		\draw  (-6.8, 7.5) node [above] {\Large{$\ell_{\v(-H)}$}};
		
		\draw  (-1.3,-1.02) node [right] {$b_1^{\v}$};
		\draw  (-2.25, -.25) node [right] {$b_2^{\al}$};
		\draw  (-4.8,3.85) node [below] {$b_1^*$};
		\draw  (-5.1, 5) node [left] {$b_1^{\al}$};
		\draw  (-5.5,6.08) node [right] {$b_1^{\v(-H)}$};
		\draw  (-5.68,6.8) node [right] {$\tilde{b}_1$};
		\draw  (-5.9, 6.4) node [left] {$\Pi(\v(-H))$};
		
		\draw [color=blue, thick] (0, 0)--(0, -1.5);
		\begin{scriptsize}
		\fill (-1.7, -1.28) circle (2pt);
		\fill (5.45,3.95) circle (2pt);
		\fill (-6., 6.6) circle (2pt);
		\fill [color=blue] (0, 0) circle (2pt);
		\fill (4.3,3.1) circle (1.5pt);
		\fill (4.75,4.1) circle (1.5pt);
		\fill (3.63,1.8) circle (1.5pt);
		
		\fill (-1.34,-1.02) circle (1.5pt);
		\fill (-2.25, -.25) circle (1.5pt);
		\fill (-4.65,3.85) circle (1.5pt);
		\fill (-5.1, 5) circle (1.5pt);
		\fill (-5.5,6.08) circle (1.5pt);
		\fill (-5.68,6.55) circle (1.5pt);
		
		\end{scriptsize}
		
		\end{tikzpicture}
		
		\caption{The lines $\ell^*$, $\ell_{\v}$, $\ell_{\v(-H)}$, $\ell_{\alpha}$ and $\widetilde{\ell}$.
		}
		
		\label{projetcion}
		
	\end{centering}
\end{figure}
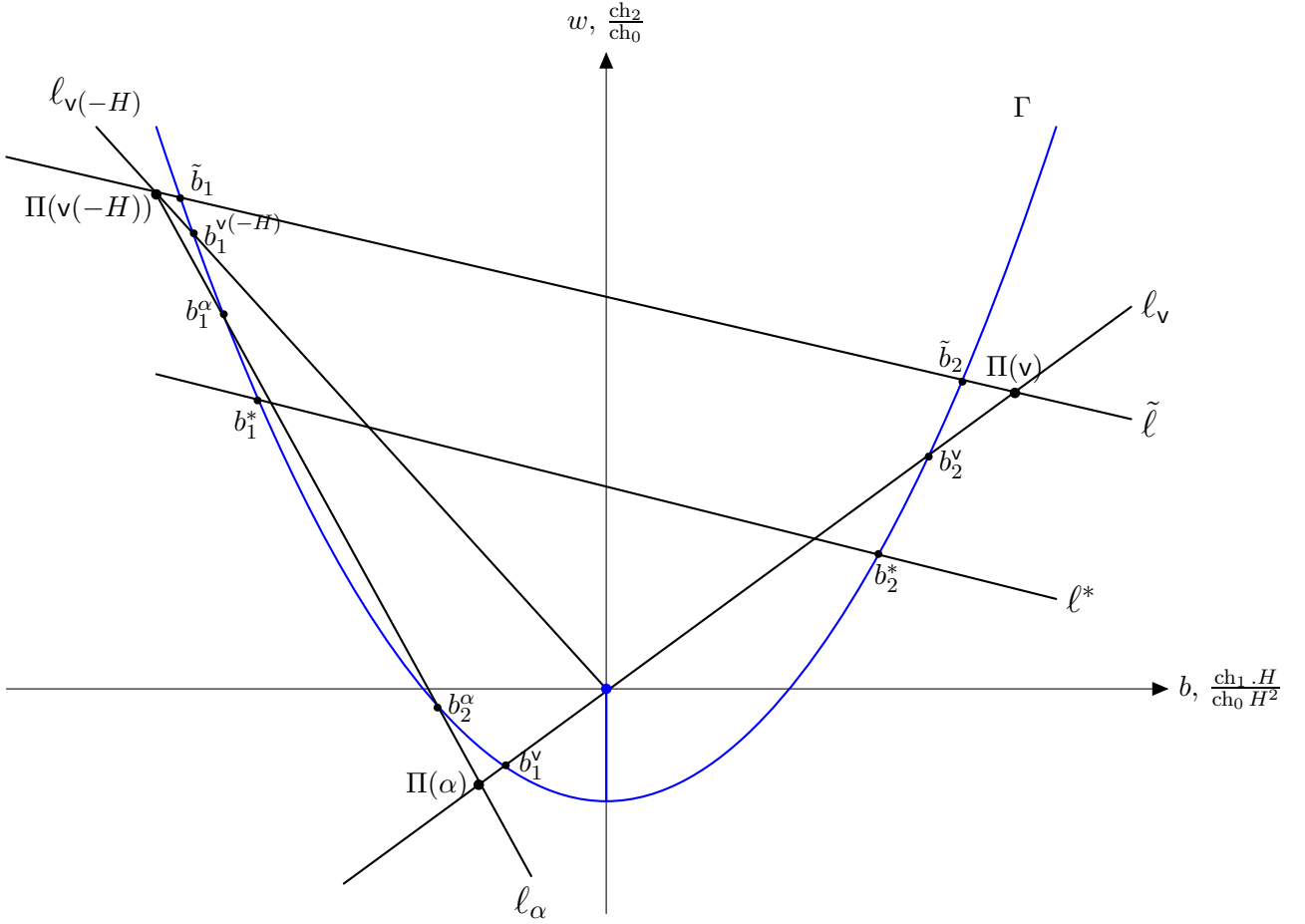

\begin{Prop} \label{prop.all bounds} 
	Suppose $H^2= 2g-2 \gg 0$.
	\begin{enumerate*}
		\item The origin point $(0,0)$ lies below $\ell^*$ and 
		\begin{equation*}
		\frac{k-r}{r}< b_1^* < b_2^* < \frac{k}{r}.
		\end{equation*}
		\item The line $\widetilde{\ell}$ passing through $\Pi(\v)$ and $\Pi(\v(-H))$ is parallel to $\ell^*$ and lies above it. It intersects the curve $\Gamma$ at two points with $b$-values $\widetilde b_1 < \widetilde b_2$ such that 
		\begin{equation*}
		\frac{k-r}{r} \leq \widetilde b_1 < b_1^* \qquad \text{and}\qquad 
		b_2^* < \widetilde b_2 \leq  \frac{k}{r}.
		\end{equation*}
		\item The line $\ell_{\v}$ passing through $\Pi(\v)$ and the origin $(0, 0)$ intersects $\Gamma$ at two points (except the origin) with $b$-values $b_1^{\v} < b_2^{\v}$ satisfying 
		\begin{equation*}
		 b_1^{\v} < -\frac{k}{s} + \frac{1}{s(s-1)} \qquad \text{and} \qquad b_2^* <b_2^{\v} 
		 \leq \frac{k}{r}.
		\end{equation*}
		\item The line $\ell_{\al}$ passing through $\Pi(\v(-H))$ and $\Pi(\al)$ intersects $\Gamma$ at two points with $b$-values $b_1^{\al} <b_2^{\al}$ so that 
		\begin{equation*}
		 \frac{k-r}{r} \leq 
		  b_1^{\al} < b_1^* \qquad \text{and} \qquad -\frac{k}{s} - \frac{1}{s(s-1)} < b_2^{\al}. 
		\end{equation*} 
		\item The line $\ell_{\v(-H)}$ passing through $\Pi(\v(-H))$ and the origin intersects the curve $\Gamma$ at two points (except the origin) with $b$-values $b_1^{\v(-H)} < 0 <b_2^{\v(-H)}$ such that 
		\begin{equation*}
		\frac{k-r}{r} \leq b_1^{\v(-H)} 
		< b_1^*.
		\end{equation*}  
	\end{enumerate*}
\end{Prop}

\begin{Lem}\label{lem-walls above ell start in section 3}
	Let $\ell^*$ intersects the vertical lines $b = \frac{k}{r}$ and $b = \frac{k-r}{r}$ at $p_1 = (\frac{k}{r}, w_1)$ and $p_2= (\frac{k-r}{r}, w_2)$, respectively. The line $\ell_1$ that passes through $p_1$ and the origin intersects the vertical line $b = \frac{k}{r} -\frac{1}{r(r-1)}$ at a point inside $U$. Similarly, the line $\ell_2$ through $p_2$ and the origin intersects $b = \frac{k-r}{r} + \frac{1}{r(r-1)}$ at a point inside $U$.    
\end{Lem}
We postpone the proof of Proposition \ref{prop.all bounds} and Lemma \ref{lem-walls above ell start in section 3} to Section \ref{section.location}. 
The next Lemma gives vertical lines on which we can rule out walls of instability for objects of certain classes. 
\begin{Lem}\label{lem-minimal}
	Take an object $E \in \cD(X)$ with $\ch_{\leq 1}(E) = (r', k'H)$ such that $\gcd(r', k') =1$. There are unique $m^{\pm}, n^{\pm} \in \mathbb{Z}$ such that $\abs{n^{\pm}} < r'$, $n^{\pm}r' > 0$ and  
	\begin{equation*}
	m^{-} r' - n^{-} k' = -1 \qquad \text{and} \qquad m^+ r' - n^+ k' = 1. 
	\end{equation*}
	We set 
	\begin{equation*}
	b_E^{\pm} \coloneqq \frac{m^{\pm}}{n^\pm} = \frac{k'}{r'} \pm \frac{1}{n^\pm r'}\,,
	\end{equation*}
	then there is no wall for $E$ crossing the vertical lines $b = b_E^{\pm}$, i.e., if $E$ is $\sigma_{b=b_E^{\pm},w}$-semistable for some $w > \Gamma(b_E^{\pm})$, then it is $\sigma_{b=b_E^{\pm},w}$-stable for any $w > \Gamma(b_E^{\pm})$.   
\end{Lem}
\begin{proof}
	The first part is trivial and the second part follows by the fact that the imaginary part 
	\begin{equation*}
	\abs{\text{Im}[Z\_{b = b_E^{\pm}, w}(E)]} = \abs{k' - \frac{m_i}{n_i} r'} = \frac{1}{\abs{n_i}}
	\end{equation*}
	is minimal, see \cite[Lemma 3.5]{feyz:effective-restriction-theorem} for more details. 
\end{proof}

Let $M_{X}(\v)$ be the moduli space of $H$-Gieseker semistable sheaves on the surface $X$ with Mukai vector $\v = (r,kH,s)$. 
\begin{Lem}\label{lem-locally-free}
	Any coherent sheaf $E \in M_{X}(\v)$ is a $\mu_H$-stable locally free sheaf.
\end{Lem}
\begin{proof}
	Since $\gcd(r, k) =1$, any $H$-Gieseker-semistable sheaf of class $\v$ is $\mu_H$-stable. Assume $E$ is not locally-free, then taking its reflexive hull gives the exact sequence
	\begin{equation*}
	E \hookrightarrow E^{\vee\vee} \twoheadrightarrow Q
	\end{equation*}
	where $Q$ is a torsion sheaf supported in dimension zero. We know the reflexive sheaf $E^{\vee \vee}$ is also slope-stable, so 
	\begin{equation*}
	-2 \leq v(E^{\vee \vee})^2 = (r, kH, s +\ch_3(Q))^2 = k^2H^2-2r(s +\ch_3(Q)) \leq k^2H^2-2rs -2r
	\end{equation*}
	which is not possible by our assumption \eqref{assumption}.
\end{proof}

 Define the object $K_E[1] \in \mathcal{D}(X)$ as the cone of the evaluation map:
\begin{equation}\label{cone}
\mathcal{O}_X^{h^0(X,E)} \xrightarrow{\text{ev}} E \rightarrow K_E[1].
\end{equation}
\begin{Prop}\label{prop.restriction}
	Let $E \in M_{X}(\v)$ be a $\mu_H$-stable vector bundle on the surface $X$. Then
	\begin{itemize}
		\item[(a)] $\Hom\big(E,E(-H)[1]\big) = 0$.
		\item[(b)] For any curve $C \in |H|$, the restriction $E|_C$ is a stable\footnote{Stability on the curve $C$ is defined as in Definition \ref{def-stability}} vector bundle on $C$ and $h^0(C,E|_C) = r+s$. 
		\item[(c)] The object $K_E$ is a $\mu_H$-stable locally free sheaf on $X$ of Mukai vector $v(K_{E}) = \al = (s, -kH, r) $ and $\Hom\big(K_E, E(-H)[1]\big) = 0$.
	\end{itemize}
\end{Prop} 
\begin{proof}
	\textbf{Step 1.} We first do wall-crossing for the bundle $E$. By \cite[Proposition 14.2]{bridgeland:K3-surfaces}, $E$ is $\sigma_{b, w}$-stable for $b< \mu_H(E)$ and $w \gg 0$. As in Lemma \ref{lem-minimal} consider the vertical line $b= b^-_E$. We know 
	\begin{equation*}
	\frac{k}{r} -\frac{1}{r} \leq\, b_E^- \,\leq \frac{k}{r}-\frac{1}{r(r-1)}.  
	\end{equation*}  
	By Proposition \ref{prop.all bounds} (c), the line $\ell_{\v}$ intersects $\Gamma$ at a point with positive $b$-value $b_2^{\v}$ satisfying
	\begin{equation*}
	\frac{k}{r}-\frac{1}{r^2(r+1)} \overset{\eqref{cond.for ellstar}}{\leq} b_2^* < b_2^{\v}. 
	\end{equation*}
	Thus the line segment $\ell_\v \cap U$ intersects the vertical line $b= b^-_E$ at a point inside $U$. Hence Lemma \ref{lem-minimal} implies that there is no wall for $E$ above $\ell_{\v}$.
	
	We have $\Hom(\cO_X, E[2]) = \Hom(E, \cO_X)^* = 0$, so 
	\begin{equation*}
	\chi(\cO_X, E) = r+s = h^0(E) -h^1(E). 
	\end{equation*} 
	Thus $h^0(E) \neq 0$ and $\cO_X$ makes a wall for $E$. This wall is indeed $\ell_\v \cap U$ for $b \in (b_1^{\v}, 0)$. Take a point $(b,w)$ on this wall and consider the evaluation map 
	\begin{equation*}
	\cO_X^{h^0(E)} \xhookrightarrow{\text{ev}} E \twoheadrightarrow K_E[1]
	\end{equation*}
	We know $\cO_X$ and $E$ are $\sigma_{b,w}$-semistable of the same $\nu\_{b,w}$-slope and $\cO_X$ is strictly $\sigma_{b,w}$-stable. Thus the map $\ev$ is injective in the abelian subcategory of $\sigma_{b,w}$-semistable objects in $\cA(b)$ of $\nu_{b,w}$-slope equal to $\nu_{b,w}(\cO_X)$. Therefore the co-kernel $K_E[1]$ is also $\sigma_{b,w}$-semistable. 
	  
	By \cite[Lemma 3.2]{feyz:mukai-program}, we know    
	\begin{equation*}
	h^0(X, E) \leq \frac{r+s}{2} + \frac{\sqrt{(r-s)^2 +k^2(2H^2 +4)}}{2} \coloneqq h
	\end{equation*}
	We claim $h < r+s+1$, i.e. 
	\begin{equation*}
     (r-s)^2 +k^2(2H^2 +4) < (r+s +2)^2.
	\end{equation*}
	This is equivalent to $k^2H^2 -2rs< 2 +2(r+s)-2k^2 $. Since $\v^2 =k^2H^2-2rs < 2r-2$ by \eqref{assumption}, we only need to show 
	\begin{equation*}
	2r-2 \leq 2+2(r+s)-2k^2
	\end{equation*}
	i.e. $-2+k^2 \leq s$ which holds by \eqref{bound for s} if $H^2 > 2r$. Therefore $h^0(E) = r+s$, $h^1(E) = 0$ and $v(K_E[1]) = -\al = (-s, kH, -r)$. 
	
	\bigskip

	\textbf{Step 2.} The next step is to examine walls for $K_E[1]$. By Proposition \ref{prop.all bounds} (c), our Brill-Noether wall $\ell_\v \cap U$ intersects $\Gamma$ at a point with $b$-value $b_1^\v < -\frac{k}{s} + \frac{1}{s(s-1)}$. We know $\gcd(s, k) = 1$, so if $k >1$, the value $b_{K_E}^{+}$ in Lemma \ref{lem-minimal} satisfies
	\begin{equation*}
	-\frac{k}{s} + \frac{1}{s(s-1)} \leq b_{K_E}^{+} \leq -\frac{k}{s} + \frac{1}{s} < 0. 
	\end{equation*}  
	Thus the vertical line $b = b_{K_E}^{+}$ intersects $\ell_\v \cap U$ at a point inside $U$. Hence $\sigma_{b,w}$-semistability of $K_E[1]$ for $(b,w) \in \ell_\v \cap U$ implies that it is $\sigma_{b = b_{K_E}^{+} ,w}$-stable for any $w > \Gamma(b_{K_E}^{+})$ by Lemma \ref{lem-minimal}.

	If $k=1$, then we show directly that $K_E[1]$ is $\sigma_{b,w}$-stable for $(b,w) \in \ell_\v \cap U$ when $b < 0$. Assume otherwise, and let $K_1$ be a $\sigma_{b,w}$-stable factor of $K_E$. There are $t_1, s_1 \in \mathbb{Q}$ such that $\ch(K_1) = s_1\ch(\cO_X) +t_1\ch(E)$, see for instance \cite[Remark 2.5]{feyz:mukai-program}. Taking $\ch_1$ implies $t_1 \in \mathbb{Z}$. Moving along the wall $\ell_{\v}$ towards the origin implies that 
	\begin{equation*}
	0 \leq \lim\limits_{b \rightarrow 0}\text{Im}[Z_{(b,w)}(K_1)] = t_1 \leq \lim\limits_{b \rightarrow 0}\text{Im}[Z_{(b,w)}(K_E[1])] =1.
	\end{equation*}   
	Thus $t_1 = 0$ or $1$, so the structure sheaf $\mathcal{O}_X$ is a subobject or a quotient of $K_E[1]$. By definition Hom$(\mathcal{O}_X,K_E[1]) = 0$ and since Hom$(E , \mathcal{O}_X) = 0$, we have Hom$(K_E[1],\mathcal{O}_X) =0$, a contradiction. Therefore $K_E[1]$ is $\sigma_{b,w}$-stable along the wall $\ell_{\v} \cap U$ for $b<0$. Openness of stability implies that it is $\sigma_{b=0, w}$-stable for $0 < w \ll 1$. Then Lemma \ref{lem-minimal} implies that $K_E[1]$ is $\sigma_{b=0, w}$-stable for any $w > 0$. 
	
	Hence in both cases, $K_E[1]$ is $\sigma_{b,w}$-stable for $b > \mu(K_E)$ and $w \gg 0$. \cite[Lemma 6.18]{macri:intro-bridgeland-stability} implies that $\cH^{-1}(K_E[1])$ is a $\mu_H$-semistable torsion-free sheaf and $\cH^0(K_E[1])$ is supported in dimension zero. Since $\gcd(s, k) =1$, $\cH^{-1}(K_E[1])$ is $\mu_H$-stable. If $\cH^{0}(K_E[1])$ is non-zero, then 
	\begin{align}\label{com}
	-2 \leq v(\cH^{-1}(K_E[1]))^2 =\ & \big(s, -kH, r+ \ch_3(\cH^0(K_E[1]))\big)^2 \nonumber\\
	=\ & v(E)^2 -2s\ch_3(\cH^0(K_E[1])) \nonumber\\
	\overset{\eqref{assumption}}{<}\ & 2r-2-2s\,.
	\end{align}
	This implies $s < r$, thus \eqref{bound for s} gives 
	\begin{equation*}
	-1+\frac{1}{r} + \frac{k^2}{2r}H^2 < r
	\end{equation*} 
	which is not possible as $H^2> 2r(r+1)$. Hence $K_E$ is a $\mu_H$-stable sheaf, and so is $\sigma_{b,w}$-stable for $b > -\frac{k}{s}$ and $w \gg 0$ by \cite[Proposition 14.2]{bridgeland:K3-surfaces}. By taking the reflexive hull and doing the same computations as in \eqref{com}, one can easily check that $K_E$ is indeed locally-free.

	\textbf{Step 3.} The final step is to analyse walls for $K_E$ and $E(-H)[1]$ when $b < \mu(K_E)$. Consider the vertical line $b = b^-_{K_E}$ as in Lemma \ref{lem-minimal}. By Proposition \ref{prop.all bounds} (d), the line $\ell_{\al}$ intersects $\Gamma$ at two points with $b$-values $b_1^{\alpha} < b_2^{\alpha}$ such that 
	\begin{equation}\label{w.2}
	b_1^{\al} < -\frac{k}{s} - \frac{1}{s} \leq b^-_{K_E} \leq -\frac{k}{s} -\frac{1}{s(s-1)}  < b_2^{\al}.  
	\end{equation} 
   The first inequality follows from 
  \begin{equation*}
  b_1^{\al} < b_1^* \overset{\eqref{cond.for ellstar}}{\leq} \frac{k-r}{r} + \frac{1}{r^2(r+1)}
  \end{equation*}
  and the point that for $g \gg 0$, we have 
	\begin{align}\label{cond--1}
	\frac{k}{r} -1 + \frac{1}{r^2(r+1)} <\ & \frac{-(k+1)}{\frac{k^2}{2r}H^2-1} \\
	 \overset{\eqref{bound for s}}{<} & -\frac{k}{s} - \frac{1}{s}.\nonumber
	\end{align}
    Thus Lemma \ref{lem-minimal} shows that by moving down along the vertical line $b= b_{K_E}^-$ we get $K_E$ is $\sigma_{b,w}$-stable for any $(b,w) \in \ell_\al \cap U$.
	
	On the other hand, $E$ is a $\mu_H$-stable vector bundle, thus \cite[Lemma 3.3]{feyz:effective-restriction-theorem} implies that $E(-H)[1]$ is $\sigma_{b,w}$-stable for $b > \frac{k-r}{r}$ and $w \gg 0$. If $k \neq r-1$, consider the vertical line $b = b^+_{E(-H)}$, then by Proposition \ref{prop.all bounds} (d),   
	\begin{equation*}
	b_1^{\al} < b_1^* \overset{\eqref{cond.for ellstar}}{\leq} \frac{k-r}{r} + \frac{1}{r(r-1)} \leq b^+_{E(-H)} \leq \frac{k-r}{r} + \frac{1}{r} \leq -\frac{1}{r}. 
	\end{equation*} 
	If $k=r-1$, then we consider the vertical line $b = -\frac{1}{r+1}$ and the same argument as in Lemma \ref{lem-minimal} shows that there is no wall for $E(-H)[1]$ crossing this vertical line. For $g \gg 0$, we have 
	\begin{align}\label{cond--2}
	 -\frac{1}{r+1} < & \frac{-(k+1)}{\frac{k^2}{2r}H^2-1}\\
	 \overset{\eqref{bound for s}}{<} & -\frac{k+1}{s} < b_2^{\alpha}\nonumber,
	\end{align}
    Hence Proposition \ref{prop.all bounds} implies the lines $\widetilde{\ell}$, $\ell_{\v(-H)}$ and $\ell_{\al}$ intersect the vertical line $b = b^+_{E(-H)}$ if $k<r-1$ or $b=-\frac{1}{r+1}$ in case $k =r-1$ at points inside $U$. Hence by Lemma \ref{lem-minimal}, $E(-H)[1]$ is $\sigma_{b,w}$-stable along these three lines and so we get the following:
    
	\noindent (a) $E$ and $E(-H)[1]$ are stable of the same $\nu\_{b,w}$-slope for $(b,w) \in \widetilde{\ell}\, \cap\, U$, therefore $\Hom(E, E(-H)[1]) = 0$ as claimed in part (a). 
	
	\noindent (b) We know $\iota_*E|_C$ lies in the exact sequence
	\begin{equation*}
	E \hookrightarrow \iota_* E|_C \twoheadrightarrow E(-H)[1]
	\end{equation*}   
	in $\cA(b =0)$. Applying the same argument as in \cite[Corollary 4.3]{feyz:effective-restriction-theorem} implies that $E|_C$ is stable. Since $E(-H)[1]$ and $\cO_X$ are $\sigma_{b,w}$-stable of the same $\nu_{b,w}$-slope for $(b,w) \in \ell_{\v(-H)}\, \cap\, U$, we get $\Hom(\cO_X, E(-H)[1]) = 0$.
	Thus $h^0(C, E|_C) = h^0(X, E) = r+s$. This completes the proof of (b).  
	
	\noindent (c) We have shown that $K_E$ is a $\mu_H$-stable locally free sheaf. Moreover $E(-H)[1]$ and $K_E$ are $\sigma_{b,w}$-stable of the same slope with respect to $(b,w) \in \ell_\al\, \cap\, U$, so $\Hom(K_E, E(-H)[1]) = 0$ as claimed in part (c). 
\end{proof}

\subsection*{Wall-crossing for the push-forward of vector bundles on curves} As before $C \xhookrightarrow{\iota} X$ is a curve in the linear system $|H|$. Let $F$ be a semistable vector bundle on $C$ of rank $r$ and degree $kH^2$. Then $\iota_*F$ is of Mukai vector  
\begin{equation*}
v(\iota_* F) = \left(0, \ rH, \ kH^2 -\frac{r}{2}H^2  \right). 
\end{equation*} 
Semistability of $F$ implies that $\iota_* F$ is $H$-Gieseker-semistable, so it is $\sigma_{b,w}$-semistable for $w \gg 0$ \cite[Proposition 14.2]{bridgeland:K3-surfaces}. The walls of instability for $\iota_*F$ are parallel lines of slope $H^2\left(\frac{k}{r} -\frac{1}{2}\right)$.

 \begin{Prop}\label{prop-wall}
 	Let $\ell$ be a wall for $\iota_* F$ which lies above or on the line $\ell^*$ (see Definition \ref{def-ellstar}). Let $F_1 \hookrightarrow \iota_*F \twoheadrightarrow F_2$ be a destabilising sequence along $\ell$ with $\nu\_{b,w}(F_1) = \nu\_{b,w}(\iota_* F)$ for $(b,w) \in \ell \cap U$ and $\nu\_{b,w^-}(F_1) > \nu\_{b,w^-}(\iota_*F)$ for $(b,w^-)\in U$ below $\ell \cap U$. Then 
 	\begin{enumerate*}
 		\item $F_1$ is a $\mu_H$-stable sheaf with $\ch_{\leq 1}(F_1) = (r, kH)$. 
 		\item $\cH^{-1}(F_2)$ is a $\mu_H$-stable sheaf with $\ch_{\leq 1}(\cH^{-1}(F_2)) = (r, (k-r)H)$ and $\cH^0(F_2)$ is either zero or a torsion sheaf supported in dimension zero. 
 		\item The sequence $F_1 \hookrightarrow \iota_* F \twoheadrightarrow F_2$ is the HN filtration of $\iota_*F$ with respect to $\sigma_{b=0,w}$ for $0 < w \ll 1$.  
 		\item The wall $\ell$ lies below or on $\widetilde{\ell}$. If it coincides with $\widetilde{\ell}$, then $\ch_2(F_1)$ is maximum (equal to $s-r$) and $F= F_1|_C$. 
 	\end{enumerate*} 
\end{Prop}
 \begin{proof}
The first part of the argument is similar to \cite[Proposition 4.2]{feyz:mukai-program}; we add it for completeness. Taking cohomology of the destabilising sequence gives a long exact sequence of coherent sheaves
 \begin{equation}\label{exact.12}
 0 \rightarrow \cH^{-1}(F_1) \rightarrow 0 \rightarrow \cH^{-1}(F_2) \rightarrow \cH^0(F_1) \xrightarrow{d_0} \iota_*F \xrightarrow{d_1} \cH^0(F_2) \rightarrow 0.
 \end{equation}
 Thus $\cH^{-1}(F_1) = 0$ and $\cH^0(F_1) \cong F_1$. Let $v(F_1) = \big(r',k'H,s'\big)$. If $r' =0$, then $F_1$ and $i_*F$ have the same $\nu\_{b,w}$-slope with respect to any $(b,w) \in U$, so $F_1$ will not destabilise $\iota_*F$ below the wall. Hence $r' >0$. The first step is to show that $r'=r$.
 
 Let $T(F_1)$ be the maximal torsion subsheaf of $F_1$ and $F_1/T(F_1)$ be its torsion-free part. Let $v\big(T(F_1)\big) = (0,\tilde{r}H,\tilde{s})$. Right-exactness of the underived pull-back $\iota^*$ applied to the short exact sequence 
 \begin{equation}\label{torsion-exact}
 T(F_1) \hookrightarrow F_1 \twoheadrightarrow F_1/T(F_1)
 \end{equation}
 implies that   
 \begin{equation}\label{bound for rank}
 \text{rank}(\iota^*F_1) \leq \text{rank}\big(\iota^*T(F_1)\big) + \text{rank}\big(\iota^*\big(F_1/T(F_1)\big)\big). 
 \end{equation}
 Take a point $(b,w) \in \ell \cap U$. By definition of $\cA(b)$, the sequence \eqref{torsion-exact} is an exact triangle in $\cA(b)$. Consider the composition of injections $s \colon T(F_1) \hookrightarrow F_1 \hookrightarrow \iota_*F$ in $\cA(b)$. Then the cokernel $c(s)$ in $\cA(b)$ is also a rank zero sheaf because if $\cH^{-1}(c(s)) \neq 0$, the it must be a torsion-free sheaf which is not possible. Therefore, $T(F_1)$ is a subsheaf of $\iota_*F$. Since $F$ is a vector bundle on an irreducible curve $C$, we get rank$(\iota^*T(F_1)) = \tilde{r}$. Thus \eqref{bound for rank} gives $\text{rank}(\iota^*F_1) \leq \tilde{r} + r'$. Let $v\big(\cH^0(F_2)\big) = \big(0,k''H,s''\big)$. The right-exactness of $\iota^*$ implies
 \begin{equation}\label{above}
r-k'' =  \text{rank}\big(\iota^*\ker d_1 \big) = \text{rank}\big(\iota^*\im d_0 \big)  \leq \text{rank}(\iota^*F_1) \overset{\eqref{bound for rank}}{\leq} \tilde{r}+r'.
 \end{equation} 
 Since $\ell$ lies above or on $\ell^*$, it intersects $\Gamma$ at two points with $b$-values $b_1<b_2$ such that $b_1 \leq b_1^*$ and $b_2^* \leq b_2$. Note that the point $(0, 0)$ lies below $\ell \cap U$ by Proposition \ref{prop.all bounds} (a). Therefore
 \begin{equation}\label{slopes}
 \mu^+(\cH^{-1}(F_2)) \leq b_1^* 
 \qquad \text{and} \qquad b_2^* \leq 
  \mu^-(F_1).
 \end{equation}
 This implies
 \begin{align*}
 \dfrac{r-k''-\tilde{r}}{r'} = \mu_H\big(F_1/T(F_1)\big) -  \mu_H\big(\cH^{-1}(F_2)\big)
 \geq b_2^*-b_1^* 
 \end{align*}
 and so
 \begin{equation}\label{below}
 r' \leq \frac{r-k''-\tilde{r}}{b_2^*-b_1^*} < (r-k''-\tilde{r})+1. 
 \end{equation}
 Here the right hand side inequality comes from
 \begin{equation*}
 \frac{r-k''-\tilde{r}}{r-k''-\tilde{r}+1}\leq \frac{r}{r+1} < 1-\frac{1}{r^2(r+1)} \overset{\eqref{cond.for ellstar}}{\leq} b_2^*-b_1^* \,.
 \end{equation*}
 Comparing \eqref{above} with \eqref{below} implies $r' = r-k'' -\tilde{r}$. 
 Then \eqref{slopes} implies
 \begin{equation}\label{a.1}
 b_2^* \leq\ \mu_H(F_1/T(F_1)) = \frac{k'- \tilde{r}}{r-k''-\tilde{r}} 
 = \frac{\frac{\ch_1(\cH^{-1}(F_2)).H}{H^2} + r -k'' -\tilde{r}}{r-k'' -\tilde{r}}\ \leq b_1^* + 1.
 \end{equation}
 Moreover, Proposition \ref{prop.all bounds} (a) and \eqref{cond.for ellstar} give 
 \begin{equation}\label{a.2}
 b_2^* < \frac{k}{r} < b_1^*+1 \qquad \text{and} \qquad b_1^*+1-b_2^* \leq \frac{2}{r^2(r+1)}.
 \end{equation}
 Thus comparing \eqref{a.1} and \eqref{a.2} shows 
 \begin{equation*}
\abs{\frac{k'-\tilde{r}}{r-k''-\tilde{r}} - \frac{k}{r}} <  \frac{1}{r(r+1)}
 \end{equation*}
 which is possible only if we have equality of two fractions  $\frac{k'-\tilde{r}}{r-k''-\tilde{r}} = \frac{k}{r}$. Since $\gcd(r, k) =1$, we get $k''= \tilde{r} = 0$, $(r', k') = (r, k)$ and $v(T(F_1)) = (0,0,\tilde{s})$. However $T(F_1)$ cannot be a skyscraper sheaf because $T(F_1)$ is a subsheaf of $\iota_*F$ as explained above. Thus $T(F_1) = 0$ and $F_1$ is torsion-free. Then \eqref{slopes} and \eqref{cond.for ellstar} give
 \begin{equation*}
 \frac{k}{r} -\frac{1}{r^2(r+1)} \leq b_2^* \leq \mu_H^-(F_1) \leq \mu_H(F_1) =\frac{k}{r}
 \end{equation*}
 which implies $F_1$ is $\mu_H$-stable. This completes the proof of (a).
 
 Since $k''= 0$, we have $\ch_{\leq 1}(\cH^{-1}(F_2)) = (r, (k-r)H)$, so the first inequality in \eqref{slopes} implies 
 \begin{equation*}
  \frac{k-r}{r} \leq \mu_H^+(\cH^{-1}(F_2)) \leq b_1^* \leq  \frac{k-r}{r} + \frac{1}{r^2(r+1)}. 
 \end{equation*}
 Therefore $\cH^{-1}(F_2)$ is a $\mu_H$-stable sheaf as $\gcd(r, k-r) =1$. This completes the proof of part (b).  
 
 Since $F_1$ is $\mu_H$-stable, we get 
 \begin{equation*}
 v(F_1)^2 = k^2H^2 -2r(r+\ch_2(F_1)) \geq -2. 
 \end{equation*}
 Thus $\ch_2(F_1) \leq s-r$. We know $\ell$ passes through $\Pi(F_1)$, so if $\ch_2(F_1) =s-r$, then $\ell$ coincides with $\widetilde{\ell}$ and if $\ch_2(F_1) < s-r$, it lies below $\widetilde{\ell}$. 

 We claim $F_1$ is $\sigma_{b=0, w}$-stable for any $w >0$. 
 Consider the vertical line $b= b_{F_1}^-$ as in Lemma \ref{lem-minimal}; it satisfies
 \begin{equation}\label{b}
 0 \leq \frac{k}{r} -\frac{1}{r} < b_{F_1}^- \leq \frac{k}{r} -\frac{1}{r(r-1)}\,. 
 \end{equation}

Let $\ell_{F_1}$ be the line passing through $\Pi(F_1)$ and the origin. The line segment $\ell_{F_1} \cap U$ lies above $\ell_1 \cap U$ considered in Lemma \ref{lem-walls above ell start in section 3}, see Figure \ref{fig.vertical line}. Thus combining \eqref{b} with Lemma \ref{lem-walls above ell start in section 3} implies that $\ell_{F_1} \cap U$ intersects the vertical line $b=b^-_{F_1}$ at a point in the closure $\overline{U}$. We know the wall $\ell$ lies above $\ell_{F_1}$ as the origin point $(0, 0)$ lies below $\ell^*$ and so $\ell$ by Proposition \ref{prop.all bounds}(a). Thus $\ell$ intersects the vertical line $b= b_{F_1}^-$ at a point $(b_{F_1}^{-}, w)$ inside $U$, so semistability of $F_1$ along the wall and Lemma \ref{lem-minimal} imply that $F_1$ is $\sigma_{b=b_{F_1}^-}$-stable for any $w > \Gamma(b_{F_1}^-)$. Hence the structure of walls (which all pass through $\Pi(F_1)$) implies that $F_1$ is $\sigma_{b,w}$-stable for any $(b,w)$ above $\ell_{F_1}$ when $b < \frac{k}{r}$, so in particular, it is $\sigma_{b=0, w}$-stable for any $w >0$. 

\begin{figure}[h]
	\begin{centering}
		\definecolor{zzttqq}{rgb}{0.27,0.27,0.27}
		\definecolor{qqqqff}{rgb}{0.33,0.33,0.33}
		\definecolor{uququq}{rgb}{0.25,0.25,0.25}
		\definecolor{xdxdff}{rgb}{0.66,0.66,0.66}
		
		\begin{tikzpicture}[line cap=round,line join=round,>=triangle 45,x=1.0cm,y=0.9cm]
		
		\draw  (4, 0) node [right ] {$b,\,\frac{\ch_1\!.\;H}{\ch_0H^2}$};


		\fill [fill=gray!40!white] (0,-1) parabola (2.8,4) parabola [bend at end] (-2.8,4) parabola [bend at end] (0,-1);
		
		\draw[->,color=black] (-4,0) -- (4,0);
		\draw  (0,-1) parabola (2.9,4.27); 
		\draw  (0,-1) parabola (-2.9,4.27); 
		\draw  (3 , 4.2) node [above] {$\Gamma$};

		\draw[->,color=black] (0,-1.5) -- (0,4.7);
		\draw  (1, 4.1) node [above ] {$w,\,\frac{\ch_2}{\ch_0}$};
		
		\draw [dashed, color=black] (0, 0) -- (3.5, 2);
		\draw [dashed, color=black] (0, 0) -- (3.5, .95);

		\draw [dashed, color=black] (2.3, 0) -- (2.3, 4);
		\draw [dashed, color=black] (1., -1) -- (1., 4);
		\draw  (1.2, -1) node[below] {$b= b_{F_1}^-$};
		
		\draw [color=black] (-3.7, 3.4) -- (3.2, 1);
		\draw [color=black] (-3.7, 2.7) -- (3.2, .3);
		
		\draw [color=white] (0, 0) -- (0, -1);
		
		\draw  (3.5, .95) node[right] {$\ell_1$};
		\draw  (3.5, 2) node[right] {$\ell_{F_1}$};
		
		\draw  (-3.7, 3.45) node[left] {$\ell$};
		\draw  (-3.7, 2.75) node[left] {$\ell^*$};
		
		\draw  (2.3, 0) node[below] {$\frac{k}{r}$};

		\draw  (-1, 3) node [above] {\Large{$U$}};
		\draw (-3., 2.9) node [left] {$\ell_2$};
		\draw  (-3.4, 4.1) node [left] {$\ell_{F_2}$};
		
		\draw  (-2.7, 0) node[below] {$\frac{k-r}{r}$};
		
		\draw  (3., 2) node[above] {\small $\Pi(F_1)$};
		\draw[->,color=black] (2.3, 1.3) -- (2.8, 2.15);
		
	     \draw  (-3.55, 2) node[below] {\small $\Pi(F_2)$};
		\draw[->,color=black] (-2.7, 3.05) -- (-3.5, 2);
		
		\draw [dashed, color=black] (-2.7, 0) -- (-2.7, 4);
		\draw [dashed, color=black] (0, 0) -- (-3.1, 2.7);
		\draw [dashed, color=black] (0, 0) -- (-3.5, 4);
			\draw [dashed, color=black] (-2, -1) -- (-2, 4);
		\draw  (-1.85, -1) node[below] {$b= b_{F_2}^+$};

		\begin{scriptsize}
		\fill (2.3, 1.3) circle (1.5pt);
		\fill[color=white] (0, 0) circle (2pt);
		\fill (2.3,.63) circle (1.5pt);
		\fill (-2.7,3.05) circle (1.5pt);
		\fill (-2.7,2.35) circle (1.5pt);
		
		\end{scriptsize}
		
		\end{tikzpicture}
		
		\caption{The vertical lines $b=b_{F_1}^-$ and $b= b_{F_2}^+$}
		
		\label{fig.vertical line}
		
	\end{centering}
\end{figure}
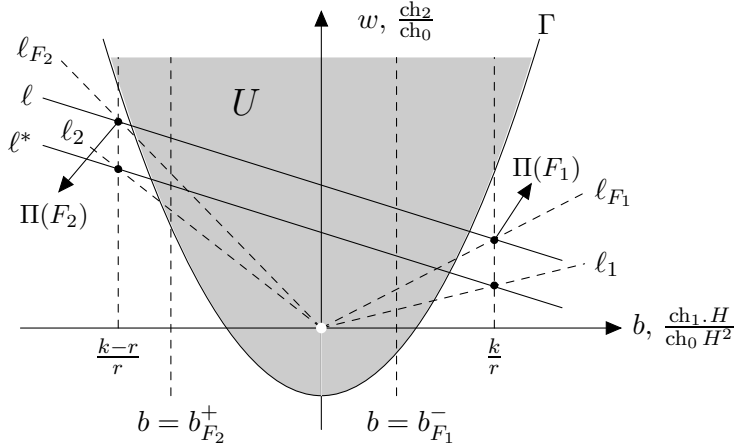

Similarly, we know 
\begin{equation*}
\frac{k-r}{r}+\frac{1}{r(r-1)} \leq\ b_{F_2}^+\ \leq \frac{k-r}{r} + \frac{1}{r} \leq 0. 
\end{equation*}
Thus by the same argument as for $F_1$, Lemma \ref{lem-walls above ell start in section 3} implies that the line $\ell_{F_2}$ passing through $\Pi(F_2)$ and the origin intersects the vertical line $b = b_{F_2}^+$ at a point inside $U$, see Figure \ref{fig.vertical line}. Thus by Lemma \ref{lem-minimal}, $F_2$ is also $\sigma_{b=0,w}$-stable for any $w >0$. We know $\nu\_{b,w}$-slope of $F_1$ is bigger than $F_2$ for $(b,w)$ below the wall $\ell$, thus the sequence $F_1 \rightarrow \iota_* F \rightarrow F_2$ is the HN filtration of $\iota_*F$ with respect to $\sigma_{0, w}$ for $0 < w \ll 1$ as claimed in part (c). 
 
  In case of equality $\ch_3(F_1) = s-r$, Proposition \ref{prop.restriction} implies that $\iota_*F_1|_C$ is a stable sheaf. We know the non-zero morphism $d_0$ in the long exact sequence \eqref{exact.12} factors via the morphism $d_0' \colon i_*F_1|_C \rightarrow i_*F$. The sheaves $i_*F_1|_C$ and $i_*F$ are both stable and have the same Mukai vector, hence $d_0'$ is an isomorphism. This completes the proof of part (d). 
\end{proof}
Proposition \ref{prop-wall} only describes walls for class $v(\iota_*F)$ which are above $\ell^*$. Instead of classifying walls below $\ell^*$, we jump over the Brill-Noether region (neighbourhood of $\Pi(\cO_X)$) and find an upper bound for the number of global sections of stable vector bundles on the curve $C$ of rank $r$ and degree $kH^2$. We first recall definition of the Harder-Narasimhan polygon.   
\begin{Def}
	Given a stability condition $\sigma_{(b,w)}$ and an object $E \in \mathcal{A}(b)$, the Harder-Narasimhan polygon of $E$ with respect to $(b,w)$ 
	is the convex hull of the points $Z_{b,w}(E')$ for all subobjects $E '\subset E$ of $E$. 
\end{Def}
If the Harder-Narasimhan filtration of $E$ is the sequence
\begin{equation*}
0 = {E}_0 \subset {E}_1 \subset .... \subset {E}_{n-1} \subset {E}_n =E,
\end{equation*}
then the points $\left\{ p_i \coloneqq Z_{b,w}({E}_i) \right\}_{i=0}^{n}$ are the extremal points of the Harder-Narasimhan polygon of $E$ on the left side of the line segment $\overline{oZ_{b,w}(E)}$, see Figure \ref{polygon figure.1}. 
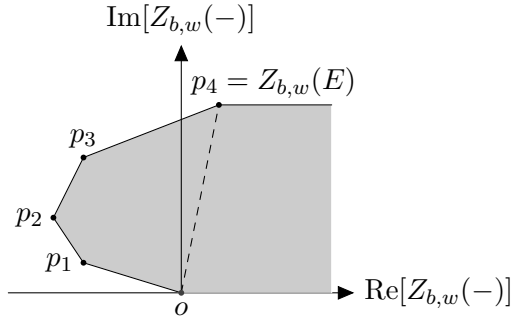
\begin{figure} [h]
	\begin{centering}
		\definecolor{zzttqq}{rgb}{0.27,0.27,0.27}
		\definecolor{qqqqff}{rgb}{0.33,0.33,0.33}
		\definecolor{uququq}{rgb}{0.25,0.25,0.25}
		\definecolor{xdxdff}{rgb}{0.66,0.66,0.66}
		
		\begin{tikzpicture}[line cap=round,line join=round,>=triangle 45,x=1.0cm,y=1.0cm]
		
		\draw[->,color=black] (-2.3,0) -- (2.3,0);
		
		\filldraw[fill=gray!40!white, draw=white] (0,0) --(-1.3,.4)--(-1.7,1)--(-1.3,1.8)--(.5,2.5)--(2,2.5)--(2,0);

		\draw [ color=black] (0,0)--(-1.3,.4);
		\draw [color=black] (-1.3,.4)--(-1.7,1);
		\draw [color=black] (-1.7,1)--(-1.3,1.8);
		\draw [color=black] (-1.3,1.8)--(.5,2.5);
		\draw [color=black] (2,2.5) -- (.5,2.5);
		\draw [color=black, dashed] (0,0) -- (.5,2.5);
		
		\draw[->,color=black] (0,0) -- (0,3.3);

		\draw (2.3,0) node [right] {Re$[Z_{b,w}(-)]$};
		\draw (0,3.3) node [above] {Im$[Z_{b,w}(-)]$};
		\draw (0,0) node [below] {$o$};
		\draw (-1.3,.4) node [left] {$p_1$};
		\draw (-1.7,1) node [left] {$p_2$};
		\draw (-1.3,1.8) node [above] {$p_3$};
		\draw (0,2.8) node [right] {$p_4 = Z_{b,w}(E)$};

		\begin{scriptsize}
		
		\fill [color=black] (-1.3,.4) circle (1.1pt);
		\fill [color=black] (-1.7,1) circle (1.1pt);
		\fill [color=black] (-1.3,1.8) circle (1.1pt);
		\fill [color=black] (.5,2.5) circle (1.1pt);
		\fill [color=uququq] (0,0) circle (1.1pt);
		
		\end{scriptsize}
		
		\end{tikzpicture}
		
		\caption{HN polygon} 
		
		\label{polygon figure.1}
		
	\end{centering}
	
\end{figure}

We want to consider the limit of HN polygon when $(b,w)$ is close to the origin. Define the function $\overline{Z} \colon K(X) \rightarrow \mathbb{C}$ as $$\overline{Z}(E) \coloneqq \lim\limits_{w \rightarrow 0} Z_{b =0, w}(E) =  -\ch_2(E) \,+\, i\, \frac{\ch_1(E).H}{H^2}.$$ 
Take an object $E \in \mathcal{A}(b=0)$ which has no subobject $E' \subset E$ in $\mathcal{A}(0)$ with ch$_1(E') = 0$, i.e. $\nu_{b,w}^+(E) < +\infty$. \cite[Proposition 3.4]{feyz:mukai-program} implies that there exists $w^* > 0$ such that the Harder-Narasimhan filtration of $E$ is a fixed sequence
\begin{equation}\label{HN}
0 = E_{0} \subset E_{1} \subset .... \subset E_{n-1} \subset  E_n=E,
\end{equation}
with respect to all stability conditions $\sigma_{0,w}$ where $0< w < w^*$. We define $\p_E$ to be the polygon with extremal points $p_i = \overline{Z}(E_i)$ and sides $\overline{p_0p_n}$ and $\overline{p_ip_{i+1}}$ for $i \in [0, n-1]$. Since $P_E$ is the limit of HN polygon when $w \rightarrow 0$, it is also a convex polygon.

\begin{Lem}\label{lem-polygon}
	Let $F$ be a rank $r$-semistable vector bundle on $C$ with degree $kH^2$ as before. The polygon $\p_{\iota_*F}$ is contained in the triangle $\triangle oz_1z_2$ where 
	\begin{equation*}
	z_1 \coloneqq \overline{Z}(\v) = -s+ r\ +i\, k \qquad \text{and} \qquad z_2 \coloneqq \overline{Z}(\iota_*F) = H^2\left(\frac{r}{2} - k\right) +\ i\, r. 
	\end{equation*}
	If $\p_{\iota_*F}$ coincides with $\triangle oz_1z_2$, then $F = E|_C$ for a vector bundle $E \in M_X(\v)$. 
\end{Lem}
\begin{figure} [h]
	\begin{centering}
		\definecolor{zzttqq}{rgb}{0.27,0.27,0.27}
		\definecolor{qqqqff}{rgb}{0.33,0.33,0.33}
		\definecolor{uququq}{rgb}{0.25,0.25,0.25}
		\definecolor{xdxdff}{rgb}{0.66,0.66,0.66}
		
		\begin{tikzpicture}[line cap=round,line join=round,>=triangle 45,x=1.0cm,y=1.0cm]

		\draw[->,color=black] (0,0) -- (0,3.5);
		\draw[->,color=black] (-3,0) -- (3,0);
		\draw[color=black] (0,0) -- (2.5,3.3);
		\draw[color=black] (0,0) -- (-1.5,1);
		\draw[color=black] (-1.5,1) -- (2.5,3.3);
		
		\draw[color=black] (0,0) -- (-.5,1);
		\draw[color=black] (.5,2) -- (2.5,3.3);
		\draw[color=black] (.5,2) -- (-.5,1);
		
		\draw (0,0) node [below] {$o$};
		\draw (3,0) node [right] {Re$[\,\overline{Z}\,] = -\ch_2$};
		\draw (0,3.5) node [above] {Im$[\,\overline{Z}\,] = \frac{H\ch_1}{H^2}$};
		\draw (-1.5,1) node [left] {$z_1$};
		\draw (2.6,3.3) node [above] {$z_2$};
		\draw (-.5,1) node [left] {$p_1$};
		\draw (0.4,2) node [above] {$p_2$};
		
		\begin{scriptsize}
		
		\fill [color=black] (0,0) circle (1.1pt);
		\fill [color=black] (2.5,3.3) circle (1.1pt);
		\fill [color=black] (-1.5,1) circle (1.1pt);
		\fill [color=black] (-.5,1) circle (1.1pt);
		\fill [color=black] (.5,2) circle (1.1pt);

		\end{scriptsize}
		
		\end{tikzpicture}
		
		\caption{The polygon $P_{i_*F}$ is inside the triangle $T$}
		
		\label{wall.4}
		
	\end{centering}
	
\end{figure}
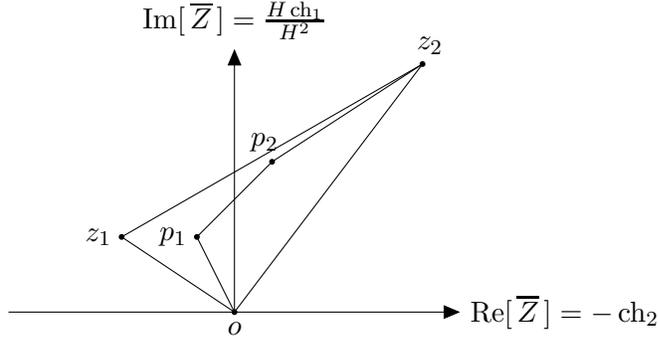	
\begin{proof}
	Consider the HN filtration \eqref{HN} for $E = \iota_* F$ with respect to $\sigma_{0, w}$ where $0 < w \ll 1$. Since $P_{\iota_* F}$ is a convex polygon, to prove the first claim we only need to show  
	\begin{equation*}
	-\dfrac{\text{Re}[\overline{Z}( E_1)]}{\text{Im}[\overline{Z}(E_1)]} \leq -\dfrac{\text{Re}[z_1]}{\text{Im}[z_1]} \qquad \text{and} \qquad
	-\dfrac{\text{Re}[z_2-z_1]}{\text{Im}[z_2-z_1]} \leq -\dfrac{\text{Re}[\overline{Z}(\iota_*F/E_{n-1})]}{\text{Im}[\overline{Z}(\iota_*F/E_{n-1})]},
	\end{equation*}
	i.e. 
	\begin{equation}\label{c}
	\frac{\ch_2(E_1)H^2}{\ch_1(E_1)H} \leq \frac{\ch_2(\v)H^2}{\ch_1(\v)H} \qquad \text{and} \qquad \frac{\ch_2(\v(-H))H^2}{\ch_1(\v(-H))H} \leq \frac{\ch_2(\iota_*F/E_{n-1})H^2}{\ch_1(\iota_*F/E_{n-1})H}. 
	\end{equation} 
	We first show that $\ch_0(E_1) >0$ and $\ch_0(\iota_*F/E_{n-1}) <0$. Taking cohomology from the exact sequence $E_1 \hookrightarrow \iota_* F \twoheadrightarrow \iota_*F/E_1$ in $\cA(b=0)$ implies that $E_1$ is a sheaf. If $E_1$ is of rank zero, then $\cH^{-1}(\iota_*F/E_1)$ is zero as it is must be a torsion-free sheaf, so $F_1$ is a subsheaf of $\iota_*F$ of bigger $\nu\_{0, w}$-slope which is not possible as $F$ is semistable\footnote{Recall that $\nu\_{b,w}$-slope of any rank zero sheaf $E$ is equal to $\frac{\ch_2(E)H^2}{\ch_1(E)H^2}$ which is independent of $(b,w)$.}, thus $\ch_0(E_1) > 0$. Similarly, the exact triangle $E_{n-1} \hookrightarrow \iota_*F \twoheadrightarrow \iota_*F/E_{n-1}$ implies $\ch_0(E_{n-1}) >0$ and so $\iota_*F/E_{n-1}$ is of negative rank. Since $E_1, \iota_*F/E_{n-1} \in \cA(b=0)$ we get
	\begin{equation*}
	0 \leq \frac{1}{\ch_0(E_1)}\text{Im}\left[Z_{0, w}(E_1)\right] = \mu(E_1) \qquad \text{and} \qquad  \frac{1}{\ch_0(\iota_*F/E_{n-1})}\text{Im}\left[Z_{0, w}(\iota_*F/E_{n-1})\right] \leq 0
	\end{equation*}
	Therefore \eqref{c} is equivalent to the claim that $\Pi(E_1)$ lies below or on the line $\ell_{\v}$ and $\Pi(\iota_*F/E_{n-1})$ lies below or on $\ell_{\v(-H)}$, see Figure \ref{projetcion.2}.
	\begin{figure}[h]
		\begin{centering}
			\definecolor{zzttqq}{rgb}{0.27,0.27,0.27}
			\definecolor{qqqqff}{rgb}{0.33,0.33,0.33}
			\definecolor{uququq}{rgb}{0.25,0.25,0.25}
			\definecolor{xdxdff}{rgb}{0.66,0.66,0.66}
			
			\begin{tikzpicture}[line cap=round,line join=round,>=triangle 45,x=1.0cm,y=0.9cm]
			
			\draw[->,color=black] (-4,0) -- (4,0);
			\draw  (4, 0) node [right ] {$b,\,\frac{H\ch_1}{H^2\ch_0}$};


			
			\draw  (0,-1) parabola (3.1,4.27); 
			\draw  (0,-1) parabola (-3.1,4.27); 
			\draw  (3.1 , 4.2) node [right] {$\Gamma$};

			\draw[->,color=black] (0,-1.5) -- (0,4.7);
			\draw  (0, 4.5) node [above] {$w,\,\frac{\ch_2}{\ch_0}$};
			
			\draw [color=black] (-3.8, 3) -- (3.5, 1);
			\draw [color=black] (-1, -.83) -- (3.5, 2.9);
			\draw [color=black] (-3.8, 4.3) -- (1, -1.1);
			
			\draw  (2.7, 2.15) node[right] {\small{$\Pi(\v)$}};
			\draw  (-3.35, 3.8) node[left] {\small{$\Pi(\v(-H))$}};
			\draw  (3.5, 1) node[right] {{$\ell^*$}};
			\draw  (2, 1) node[right] {\small{$b_2^*$}};
			\draw  (-2.55, 2.5) node[left] {\small{$b_1^*$}};
			
			\draw  (3.5, 2.9) node[right] {{$\ell_{\v}$}};
			\draw  (-3.9, 4.3) node[above] {{$\ell_{\v(-H)}$}};
			\begin{scriptsize}
			\fill (2.7, 2.23) circle (1.5pt);
			\fill (-3.45,3.9) circle (1.5pt);
			\fill (2.1,1.38) circle (1.5pt);
			\fill (-2.6,2.68) circle (1.5pt);
			
			\end{scriptsize}
			
			\end{tikzpicture}
			
			\caption{Walls for $\iota_*F$}
			
			\label{projetcion.2}
			
		\end{centering}
	\end{figure}
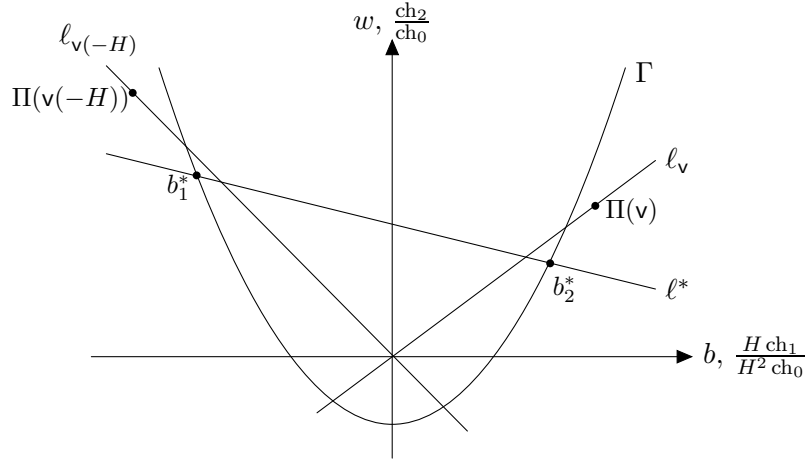
	
	If there exists a wall for $\iota_*F$ above or on $\ell^*$, Proposition \ref{prop-wall} gives a complete description of the HN factors of $\iota_*F$ with respect to $\sigma_{0, w}$ for $0 < w \ll 1$ and so the claim follows. Thus we may assume there is no wall above or on $\ell^*$ for $\iota_*F$ and so $\iota_*F$ is $\sigma_{b,w}$-semistable for any $(b,w) \in U$ above $\ell^*$. 
	
	Consider the line $\ell$ parallel to $\ell^*$ passing through $\Pi(E_1)$. We know $\iota_*F$ and $E_1$ have the same $\nu\_{b,w}$-slope for $(b,w) \in \ell \cap U$ and $E_1$ has bigger slope then $\iota_*F$ below the line as the slope function $\nu\_{b,w}(E_1)$ is a decreasing function with respect to $w$. Thus $\iota_*F$ is $\sigma_{b,w}$-unstable below $\ell$, hence $\ell$ and so $\Pi(E_1)$ lie below $\ell^*$. On the other hand, $\sigma_{0, w}$-semistability of $E_1$ implies that $\Pi(E_1)$ does not lie above $\Gamma$, so $\Pi(E_1)$ lies below the line passing through the origin and $(b_2^*, \Gamma(b_2^*))$. Then Proposition \ref{prop.all bounds}(c) implies $\Pi(E_1)$ lies below $\ell_{\v}$ as claimed. A similar argument shows that $\Pi(\iota_*F/E_{n-1})$ lies below the line passing through the origin and $(b_1^*, \Gamma(b_1^*))$ and so it lies below $\ell_{\v(-H)}$ by Proposition \ref{prop.all bounds}(e). This completes the proof of \eqref{c} and so $P_{\iota_*F}$ is contained in the triangle $\triangle oz_1z_2$. 
	

    If HN polygon coincides with the triangle $\triangle oz_1z_2$, then we have equality in both inequalities of \eqref{c}, i.e. $\Pi(E_1)$ lies on $\ell_\v$ and $\Pi(\iota_*F/E_{n-1})$ lies on $\ell_{\v(-H)}$. Since $\Pi(E_1)$ cannot lie inside $U$, it lies above or on $\ell^*$. Thus there exits a wall for $\iota_*F$ above or on $\ell^*$ 
    and so Proposition \ref{prop-wall} implies $v(F_1) = \v$ and $F = E_1|_C$ as claimed.   
\end{proof}
Let $p_i = \overline{Z}(E_i)$ be the extremal points of $\p_{\iota_*F}$ where $E_i$'s are the factor in the HN filtration of $E = \iota_*F$ as in \eqref{HN}. Then \cite[Proposition 3.4]{feyz:mukai-program} implies that 
\begin{equation}\label{bound}
h^0(X,\iota_*F) \leq \dfrac{\chi(X, \iota_*F)}{2}  + \dfrac{1}{2} \sum_{i=1}^{n} \lVert \overline{p_ip_{i-1}} \rVert
\end{equation}  
where $\lVert . \rVert$ is the non-standard norm defined in \cite[Equation (3.2)]{feyz:mukai-program}, i.e. $\lVert x+ iy \rVert = \sqrt{ x^2 + (2H^2+4) y^2}$.  


\begin{Prop}\label{prop-polygon}
	Let $F$ be a rank $r$-semistable vector bundle on $C$ with degree $kH^2$ as before. If the polygon $\p_{\iota_*F}$ lies strictly inside $\triangle oz_1z_2$, then $h^0(C, F) < r+s$.  
\end{Prop}
\begin{proof}
Since the extremal points of $\p_{\iota_*F}$ have integral coordinates, $\p_{\iota_*F}$ lies inside $oz_0'z_1'z_2'z_2$ where $z_0' = (-s+r)\frac{k-1}{k} +i\, (k-1)$, $z_1' = -s+r+1 +i\, k$ and 
\begin{equation*}
z_2' = (-kH^2 + \frac{r}{2}H^2 + s-r)\frac{1}{r-k} -s+r + i (k+1), 
\end{equation*}
see Figure \ref{fig.polygon}.

	\begin{figure}[h]
	\begin{centering}
		\definecolor{zzttqq}{rgb}{0.27,0.27,0.27}
		\definecolor{qqqqff}{rgb}{0.33,0.33,0.33}
		\definecolor{uququq}{rgb}{0.25,0.25,0.25}
		\definecolor{xdxdff}{rgb}{0.66,0.66,0.66}
		
		\begin{tikzpicture}[line cap=round,line join=round,>=triangle 45,x=1.0cm,y=1.0cm]

		\draw[->,color=black] (0,0) -- (0,3.7);
		\draw[->,color=black] (-2,0) -- (5.4,0);
		
		\draw[color=black] (0,0) -- (5,3.5);
		\draw[color=black] (0,0) -- (-1.2,1);
		\draw[color=black] (-1.2,1) -- (5,3.5);
		
		\draw[color=black] (-.6,.5) -- (-.8,1);
		\draw[color=black] (1.3,2) -- (5,3.5);
		\draw[color=black] (-.2, 1.4) -- (-.8,1);
		
		
		\draw (0,0) node [below] {$o$};
		\draw (5.4,0) node [right] {Re$[\,\overline{Z}\,] = -\ch_2$};
		\draw (0,3.6) node [above] {Im$[\,\overline{Z}\,] = \frac{H\ch_1}{H^2}$};
		\draw (-1.15,1) node [left] {$z_1$};
		\draw (-.75,.95) node [right] {$z_1'$};
		\draw (5,3.5) node [above] {$z_2$};
		\draw (-.2, 1.4) node [above] {$z_2'$};
	    \draw (-.56, .4) node [left] {$z_0'$};

		\begin{scriptsize}
		
		\fill [color=black] (0,0) circle (1.1pt);
		\fill [color=black] (5,3.5) circle (1.1pt);
		\fill [color=black] (-.6,.5) circle (1.1pt);
		\fill [color=black] (-1.2,1) circle (1.1pt);
		\fill [color=black] (-.8,1) circle (1.1pt);
		\fill [color=black] (-.2,1.4) circle (1.1pt);
		
		\end{scriptsize}
		
		\end{tikzpicture}
		
		\caption{$\p_{i_*F}$ is inside the polygon $oz_0'z_1'z_2'z_2$}
		
		\label{fig.polygon}
		
	\end{centering}
	
\end{figure}
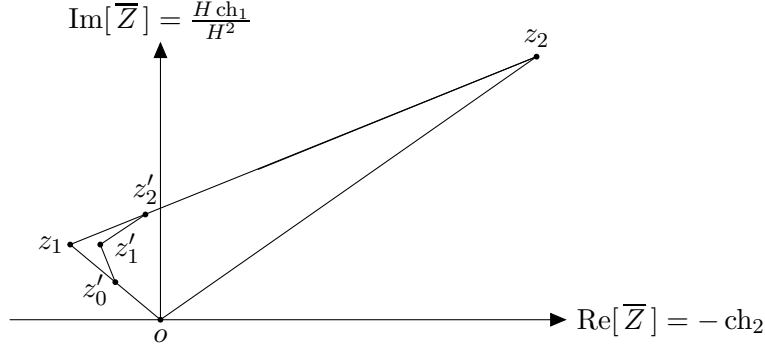
	Convexity of $\p_{\iota_*F}$ and the polygon $oz_0'z_1'z_2'z_2$ gives $$\sum_{i=1}^{n} \lVert \overline{p_ip_{i-1}} \rVert \leq \lVert \overline{oz_0'} \rVert + \lVert \overline{z_0'z_1'} \rVert + \lVert \overline{z_1'z_2'} \rVert +  \lVert \overline{z_2'z_2} \rVert  \eqqcolon Q_{\text{in}}.$$ 
 Let $Q_{\text{out}} \coloneqq \lVert \overline{oz_1} \rVert + \lVert \overline{z_1z_2} \rVert$, then 
\begin{equation*}
Q_{\text{out}}-Q_{in} =  \lVert \overline{z_0'z_1} \rVert - \lVert \overline{z_0'z'_1} \rVert + \lVert \overline{z_1z_2'} \rVert - \lVert \overline{z'_1z'_2} \rVert.
\end{equation*}
Define $\epsilon\_{\text{out}} \coloneqq \frac{1}{2}\chi(\iota_*F) + \frac{1}{2}Q_{\text{out}} - (r+s)$. Assume for a contradiction that $r+s \leq h^0(F)$, then \eqref{bound} gives
\begin{equation*}
\dfrac{1}{2}\chi(\iota_*F) + \dfrac{1}{2}Q_{\text{out}} - \epsilon\_{\text{out}}\; = \; r+s\; \leq h^0(F) \leq \dfrac{\chi(\iota_*F)}{2} + \dfrac{1}{2} \sum_{i=1}^{n}\lVert \overline{p_ip_{i-1}}\rVert  \leq \dfrac{1}{2}\chi(\iota_*F) + \dfrac{1}{2}Q_{\text{in}}.  
\end{equation*} 
Thus $Q_{\text{out}}-Q_{\text{in}} \leq 2\,\epsilon\_{\text{out}}$ but we show for $g \gg 0$ it does not hold.  
We know
\begin{align}
2\epsilon\_{\text{out}} \coloneqq \ & kH^2 - \frac{r}{2}H^2 + Q_{\text{out}} -2(r+s)\nonumber\\
= \ & \sqrt{(s-r)^2 + (2H^2+4)k^2} + \sqrt{\left(-kH^2 + \frac{r}{2}H^2+s-r\right)^2+ (2H^2+4)(r-k)^2  }
\nonumber \\
& -(s+r) - (-kH^2 + \frac{r}{2}H^2+s+r)\nonumber\\
=\ & \frac{-4rs + (2H^2+4)k^2}{\sqrt{(s-r)^2 + (2H^2+4)k^2} + (s+r)} \,+ \label{aa.1}\\ 
& \frac{-4r(-kH^2 + \frac{r}{2}H^2+s)+ (2H^2+4)(r-k)^2}{\sqrt{(-kH^2 + \frac{r}{2}H^2+s-r)^2+ (2H^2+4)(r-k)^2  } + (-kH^2 + \frac{r}{2}H^2+s+r)}\,.\label{aa.2}
\end{align}
By our assumption \eqref{assumption}, $2k^2H^2 -4rs +4k^2 \geq -4+4k^2 \geq 0$ so 
\begin{equation*}
(s+r)^2 \leq (s+r)^2-4rs+k^2(2H^2+4) = (s-r)^2 + k^2(2H^2+4). 
\end{equation*}
Thus $\eqref{aa.1} \leq \frac{-4rs + (2H^2+4)k^2}{2(r+s)}$. 

The numerator of \eqref{aa.2} is $2k^2H^2 -4rs +4(r-k)^2 \geq -4+4(r-k)^2 \geq 0$ by \eqref{assumption}, so   
\begin{align*}
\left(-kH^2 + \frac{r}{2}H^2+s+r\right)^2 \leq\ & \left(-kH^2 + \frac{r}{2}H^2+s+r\right)^2 + 2k^2H^2 -4rs +4(r-k)^2\\
 =\ & \left(-kH^2 + \frac{r}{2}H^2+s-r\right)^2 + (2H^2+4)(r-k)^2.
\end{align*}
This implies $\eqref{aa.2} \leq \frac{2k^2H^2 -4rs+4(r-k)^2}{2(-kH^2 + \frac{r}{2}H^2+s+r)}$, thus 
\begin{align*}
2\epsilon\_{\text{out}} < \frac{k^2(H^2+2) -2rs}{r+s} + \frac{k^2H^2 -2rs+2(r-k)^2}{-kH^2 + \frac{r}{2}H^2+s+r} \eqqcolon M_1
\end{align*}
When $H^2=2g-2 \rightarrow +\infty$, we know $s \rightarrow \frac{k^2}{2r}H^2 \rightarrow +\infty$ by \eqref{bound for s}, so one gets $M_1 \rightarrow 0$. On the other hand,
\begin{align}\label{l-lin}
Q_{\text{out}} -Q_{\text{in}} = \ & \sqrt{\left( \frac{1}{k}(s-r) \right)^2+ (2H^2+4) } - \sqrt{\left(\frac{1}{k}(s-r)  -1 \right)^2 + (2H^2+4) }\ \ + \\
& \sqrt{\left(\frac{1}{r-k}\left(-kH^2 + \frac{r}{2}H^2+s-r\right) \right)^2+ (2H^2+4)  }\ \  -
\nonumber\\
& \sqrt{\left(\frac{1}{r-k}\left(-kH^2 + \frac{r}{2}H^2+s-r\right) -1\right)^2+ (2H^2+4)  } \nonumber 
\end{align}
By \eqref{bound for s}, when $H^2 \rightarrow +\infty$, we have $\frac{1}{k}(s-r) >1$ and 
\begin{align*}
\frac{1}{r-k}\left(-kH^2 + \frac{r}{2}H^2+s-r\right)\overset{\eqref{bound for s}}{>}\ & \frac{1}{r-k}\left(-kH^2 + \frac{r}{2}H^2+\frac{k^2}{2r}H^2-r -1+\frac{1}{r}\right)\\
= \ & \frac{r-k}{2r}H^2 + \frac{1}{r-k}\left(-r -1+\frac{1}{r}\right)\\
> \ & 1\,.
\end{align*}
Thus \eqref{l-lin} gives 
\begin{align*}
Q_{\text{out}} -Q_{\text{in}} > \ & \frac{\frac{s-r}{k} -\frac{1}{2}}{\sqrt{ \frac{1}{k^2} (s-r)^2+ (2H^2+4) }} + \frac{\frac{-kH^2+\frac{r}{2}H^2+s-r}{r-k} -\frac{1}{2} }{ \sqrt{\frac{1}{(r-k)^2}(-kH^2 + \frac{r}{2}H^2+s-r)^2+ (2H^2+4)  }}. 
\end{align*}
When $s \rightarrow \frac{k^2}{2r}H^2 \rightarrow +\infty$, the right hand side goes to $2$, thus for $g \gg 0$, we obtain
\begin{equation}\label{epsilon}
2\epsilon\_{\text{out}} < Q_{\text{out}} -Q_{\text{in}}
\end{equation}
as claimed.  
\end{proof}

\begin{proof}[Proof of Theorem \ref{thm}]
    By Proposition \ref{prop.restriction}(b), there is a well-defined map 
    \begin{align*}
    \Psi \colon M_X(\v)& \rightarrow \bn_C(r, kH^2, r+s)\\
    E & \mapsto E|_C. 
    \end{align*}
    We know the exact sequence $E \rightarrow \iota_*E|_C \rightarrow E(-H)[1]$ in $\cA(b=0)$ is the HN filtration of $\iota_*E|_C$ below the wall $\widetilde{\ell}$, so the uniqueness of HN factors implies $\Psi$ is injective. Combining Lemma \ref{lem-polygon} with Proposition \ref{prop-polygon} shows $\Psi$ is surjective.
    Thus Proposition \ref{prop.restriction}(b) implies that any vector bundle $F$ in the Brill-Noether locus $\bn_C(r,\,k(2g-2),\, r+s)$ is stable and $h^0(F) = r+s$. The Zariski tangent space to the Brill-Noether locus at the point $[F]$ is the kernel of the map 
    \begin{equation*}
    k_1 \colon \text{Ext}^1(F,F) \rightarrow \text{Hom} \big( H^0(C,F) , H^1(C,F) \big),
    \end{equation*}
    where any $f \colon F \rightarrow F[1] \in \text{Ext}^1(F,F) = \text{Hom}_{C}(F,F[1])$ goes to 
    \begin{equation*}
    k_1(f) = H^0(f) \colon \text{Hom}_{C}(\mathcal{O}_C,F) \rightarrow \text{Hom}_{C}(\mathcal{O}_C,F[1]),
    \end{equation*}
    see \cite[Proposition 4.3]{bhosle:brill-noether-loci-on-nodal-curves} for more details. Note that the proof in \cite{bhosle:brill-noether-loci-on-nodal-curves} is valid for any family of simple sheaves on a variety. 
    
    The moduli space $M_{X}(\v)$ is a (non-empty) smooth quasi-projective variety, see for instance \cite[Chapter 10, Corollary 2.1 \& Theorem 2.7]{huybretchts:lectures-on-k3-surfaces}. Hence to prove Theorem \ref{thm}, we only need to show the derivative of the restriction map 
    \begin{equation*}
    d\Psi \colon  	T_{[E]}M_{X}(\v)  \rightarrow T_{[E|_C]} \mathcal{BN},
    \end{equation*} 
    is surjective. It sends any $f \colon E \rightarrow E[1] \in \text{Hom}_X(E,E[1])$ to its restriction $\iota^*f \colon \iota^*E \rightarrow \iota^*E[1] \in \ker(k_1)$. By Proposition \ref{prop.restriction}(c), we can apply the same argument as in the proof of \cite[Theorem 1.2]{feyz:mukai-program} to show surjectivity of $d\Psi$, hence $\Psi$ is an isomorphism.    
\end{proof}

\section{location of the first wall}\label{section.location}
In this section, we prove Proposition \ref{prop.all bounds} and Lemma \ref{lem-walls above ell start in section 3} for $g \gg 0$. The first step is to control the position of the lines $\ell_{\v}$ and $\ell_{\alpha}$.  

\begin{Prop}\label{prop-b-i}
	Fix $0 < \epsilon, \epsilon'< \frac{1}{2r}$. If $g \gg 0$, then 
	\begin{enumerate*}
		\item The line $\ell_\v$ passing through $\Pi(\v)$ and the origin intersects $\Gamma$ at two points (except the origin) with $b$-values $b_1^\v< 0 < b_2^\v$ such that 
		\begin{equation*}
		b_1^\v < -\frac{k}{s} + \frac{1}{s(s-1)}  \qquad \text{and} \qquad \frac{k}{r} -	 \epsilon < b_2^\v \ . 
		\end{equation*}
		\item The line $\ell_{\alpha}$ passing through $\Pi(\v(-H))$ and $\Pi(\al)$ intersects $\Gamma$ at two points with $b$-values $b_1^\al < b_2^\al$ such that 
		\begin{equation*}
		b_1^\al < \frac{k-r}{r} + \epsilon'  \qquad \text{and} \qquad -\frac{k}{s} - \frac{1}{s(s-1)} < b_2^\al. 
		\end{equation*}
	\end{enumerate*}
\end{Prop}
\begin{proof}
	The line $\ell_\v$ which is of equation $w = \frac{s-r}{k} b$ connects the point $\Pi(\v)$ which is outside or on $\Gamma$ to the point $(0,0)$, so it intersects $\Gamma$ at two other points except the origin with $b$-values $b_1^\v < b_2^\v$. The equation of $\Gamma$ for $b \neq 0$ is $\Gamma(b) = (1+ \frac{H^2}{2})b^2 -1$, thus 
	\begin{equation*}
	b_1^\v, b_2^{\v} = \frac{\frac{s-r}{k} \pm \sqrt{(\frac{s-r}{k})^2 +2H^2+4  }}{H^2+2}.
	\end{equation*}
	We claim if $g \gg 0$, then
	\begin{equation*}
	b_1^\v < -\frac{k}{s} + \frac{1}{s(s-1)} 
	\end{equation*}
	i.e. 
	\begin{equation*}
	\frac{s-r}{k} + \frac{\big(k(s-1)-1\big)(H^2+2)}{s(s-1)} < \sqrt{\left(\frac{s-r}{k}\right)^2 +2H^2+4}\ .  
	\end{equation*}
	Taking the power 2 of both sides shows that we only need to prove
	\begin{equation*}
	\big(k(s-1)-1\big)\left(2\frac{s-r}{k}s(s-1) + \big(k(s-1)-1\big)(H^2+2) \right) < 2s^2(s-1)^2
	\end{equation*} 
	which is equivalent to 
	\begin{equation}\label{cond.1}
	-\frac{2}{k}s(s-r) + (H^2+2)\left(k^2(s-1) + \frac{1}{s-1} -2k \right) < 2rs(s-1).
	\end{equation}
	When $H^2 \rightarrow +\infty$, then $s \rightarrow \frac{k^2}{2r}H^2 \rightarrow +\infty$ by \eqref{bound for s}, so the limit of the left hand side is $(H^2)^2\,\frac{k^3}{2r} \left(k -\frac{1}{r} \right)$ but the limit of the right hand side is $(H^2)^2\,\frac{k^4}{2r}$, thus the claim follows. 
	
	\bigskip
	
	The next step is to show for $g \gg 0$, we have 
	\begin{equation}\label{cond.2}
	b_2^\v = \frac{\frac{s-r}{k} + \sqrt{(\frac{s-r}{k})^2 +2H^2+4  }}{H^2+2} > \frac{k}{r} - \epsilon\, . 	
	\end{equation}
   This holds if 
	\begin{equation*}
	\frac{1}{2} \left(\frac{k}{r} - \epsilon \right)^2(H^2+2) < \frac{s-r}{k} \left(\frac{k}{r} - \epsilon \right)+1\,.
	\end{equation*}
	The limit of the right hand side when $s \rightarrow \frac{k^2}{2r}H^2 \rightarrow +\infty$ is $\frac{k}{2r}\left(\frac{k}{r} - \epsilon \right)H^2$, so the above holds when $g \gg 0$. This completes the proof of part (a).  
	
     The line $\ell_{\alpha}$ in part (b) is of equation $w = \theta b + \beta-1$ for 
	\begin{equation}\label{theta-beta}
	\theta \coloneqq \frac{s^2-r^2 -skH^2 +\frac{rs}{2} H^2} { s(k-r) +kr}\ , \qquad 
	\beta \coloneqq \frac{r(k-r) +sk-k^2H^2 + \frac{kr}{2}H^2}{s(k-r) +kr}.  
	\end{equation}
	When $s \rightarrow \frac{k^2}{2r}H^2 \rightarrow +\infty$, we know 
	\begin{equation}\label{limit}
	\theta\ \rightarrow\ \frac{k-r}{2r}H^2  \qquad \text{and} \qquad \beta\ \rightarrow\ \frac{k-r}{k}.
	\end{equation}
    Hence if $g \gg 0$, we have 
	 \begin{equation}\label{cond.3-before}
	 \theta^2 + 2\beta(H^2+2) \geq 0.
	 \end{equation}
	Therefore $\ell_{\al}$ intersects the parabola $w = b^2(\frac{H^2}{2} +1)-1$ at two points with $b$-values
	\begin{equation*}
	b_1^{\al},\, b_2^\al = \frac{\theta \pm \sqrt{\theta^2 + 2\beta(H^2+2)}}{H^2+2}\,. 
	\end{equation*}
	We claim if $g \gg 0$, then 
	\begin{equation}\label{cond.3}
	-\frac{k}{s} -\frac{1}{s(s-1)} <  b_2^\al = \frac{\theta + \sqrt{\theta^2 + 2\beta(H^2+2)}}{H^2+2}
	\end{equation}
	which holds if 
	\begin{equation}\label{cl}
	\frac{H^2+2}{2} \left(\frac{k}{s} + \frac{1}{s(s-1)}\right)^2 < \beta - \theta\left(\frac{k}{s}+ \frac{1}{s(s-1)}\right). 
	\end{equation}
	The right hand side is equal to 
	\begin{equation*}
	\beta - \theta\left(\frac{k}{s}+ \frac{1}{s(s-1)}\right) = \frac{
\frac{H^2}{s-1}\left(k-\frac{r}{2}\right) +r(k-r) -\frac{s}{s-1} + \frac{kr^2}{s} + \frac{r^2}{s(s-1)} 	
}{s(k-r)+kr}
	\end{equation*}
	thus its limit when $s \rightarrow \frac{k^2}{2r}H^2 \rightarrow +\infty$ is equal to $\frac{1}{H^2}\left(2\frac{r^2}{k^2} + \frac{2r(r-k)}{k^4}\right)$. Since the limit of the left hand side of \eqref{cl} is equal to $\frac{2r^2}{H^2k^2}$ the claim \eqref{cl} follows.   
	
    \bigskip
    
    The other intersection point of $\ell_{\al}$ with $\Gamma$ satisfies 
    \begin{equation*}
    b_1^\al = \frac{\theta - \sqrt{\theta^2 + 2\beta(H^2+2)}}{H^2+2}  < \frac{k-r}{r} + \epsilon' 
    \end{equation*}
    if and only if 
    \begin{equation}\label{cond.4}
    \theta - (H^2+2) \left(\frac{k-r}{r} + \epsilon' \right) <  \sqrt{\theta^2 + 2\beta(H^2+2)}\,.  
    \end{equation}
    Taking power 2 from both sides shows that \eqref{cond.4} holds if 
    \begin{equation*}
    \frac{H^2+2}{2} \left(\frac{r-k}{r} - \epsilon' \right)^2 < \beta -\theta\left(\frac{r-k}{r} - \epsilon' \right) 
    \end{equation*}
    which is satisfied by \eqref{limit} for $g \gg 0$ and $\epsilon' < \frac{1}{2r}$. 
    
\end{proof}

If $g \gg 0$, \eqref{limit} implies
\begin{equation}\label{cond.5-before}
\beta -1 < 0
\end{equation}
where $\beta$ is defined as in \eqref{theta-beta}. Thus the origin point $(0,0)$ and so the line segment connecting $\Pi(\v(-H))$ to the origin lies above $\ell_{\al}$. Thus it intersects the curve $\Gamma$ at a point with $b$-value $b_1^{\v(-H)}$ satisfying 
\begin{equation}\label{b-1-v(-H)}
\frac{k-r}{r}\, \leq\,  b_1^{\v(-H)} \,\leq\, b_1^{\al} \ .  
\end{equation}
Consider the line $\widetilde{\ell}$ passing through $\Pi(\v)$ and $\Pi(\v(-H))$. It is of equation 
\begin{equation*}
w = H^2\left(\frac{k}{r} - \frac{1}{2}\right) b \,+ \frac{s}{r} -H^2\frac{k}{r} \left(\frac{k}{r} - \frac{1}{2} \right) -1. 
\end{equation*}
When $s \rightarrow \frac{k^2}{2r}H^2 \rightarrow +\infty$, we get
\begin{equation}\label{cond.5}
\frac{s}{r} -H^2\frac{k}{r} \left(\frac{k}{r} - \frac{1}{2} \right)-1 \ > \ 0. 
\end{equation} 
This implies the origin $(0, 0)$ lies below $\widetilde{\ell}$. Therefore the line segment connecting $\Pi(\v)$ to $\Pi(\v(-H))$ (which is part of $\widetilde{\ell}$) lies above $\ell_\v$ and $\ell_{\v(-H)}$, so $\widetilde{\ell}$ intersects $\Gamma$ at two points with $b$-values $\widetilde{b}_1 < \widetilde{b}_2$ satisfying
\begin{equation*}
\frac{k-r}{r} \,\leq\, \widetilde{b}_1 \,\leq\, b_1^{\v(-H)}
\qquad \text{and} \qquad
b_2^\v\, \leq\, \widetilde{b}_2 \, \leq \, \frac{k}{r}. 
\end{equation*}
Thus \eqref{b-1-v(-H)} together with Proposition \ref{prop-b-i} implies 
\begin{equation}\label{s-1}
\frac{k-r}{r} \,\leq\, \widetilde{b}_1 \,<\, \frac{k-r}{r} +\epsilon' \qquad \text{and} \qquad \frac{k}{r} -\epsilon \,<\, \widetilde{b}_2 \,\leq\, \frac{k}{r}\,. 
\end{equation} 
Thus if we assume $H^2$ is large enough, the line $\widetilde{\ell}$ intersects $\Gamma$ at two points which are arbitrary close to $\Pi(\v)$ and $\Pi(\v(-H))$. Consider the line $\ell^*$ as in Definition \ref{def-ellstar}. We set  
\begin{equation}\label{chosen epsilon}
\epsilon' = b_1^* -\frac{k-r}{r} \qquad \text{and} \qquad  \epsilon = \frac{k}{r} -b_2^*\,. 
\end{equation}
Then the line $\ell^*$ lies below the parallel line $\widetilde{\ell}$. 
\begin{Lem}
	The origin point $(0,0)$ lies below $\ell^*$ and 
	\begin{equation}\label{claim}
	\frac{k-r}{r}< b_1^* <0< b_2^* < \frac{k}{r}.
	\end{equation}
\end{Lem}
\begin{proof}
	The equation of $\ell^*$ is given by $w = H^2\left(\frac{k}{r} -\frac{1}{2}\right) b+ \alpha -1$ for some $\alpha \in \mathbb{R}$. To prove the first claim we need to show $\alpha >1$. We know the intersection points of $\ell^*$ with $\Gamma$ satisfies  
	\begin{equation*}
	b_2^*-b_1^* = \frac{2\sqrt{(H^2)^2\left(\frac{k}{r} -\frac{1}{2}\right)^2 +2\alpha(H^2+2) }}{H^2+2} \geq \frac{k}{r} - \frac{1}{r^2(r+1)} -\frac{k-r}{r} -\frac{1}{r^2(r+1)}
	\end{equation*} 
	This implies for $g \gg 0$, we have 
	\begin{equation}\label{condition.extra}
	2\alpha \geq (H^2+2) \left(\frac{1}{2} -\frac{1}{r^2(r+1)}  \right)^2 - \frac{(H^2)^2}{H^2+2} \left(\frac{k}{r} -\frac{1}{2}\right)^2 > 2.
	\end{equation}
	The second claim \eqref{claim} follows from \eqref{s-1} and the point that by the choice of $\epsilon$ and $\epsilon'$ as in \eqref{chosen epsilon}, we know $\widetilde{\ell}$ lies above $\ell^*$.     	
\end{proof}
 The above arguments completes the proof of Proposition \ref{prop.all bounds}.  
The final step is to prove Lemma \ref{lem-walls above ell start in section 3}. 
%
\begin{proof}[Proof of Lemma \ref{lem-walls above ell start in section 3}]
	The equation of the line $\ell^*$ is given by $w = H^2(\frac{k}{r} -\frac{1}{2}) b + \alpha-1$ where $\alpha = (b_2^*)^2\left(\frac{H^2}{2} +1\right) - H^2(\frac{k}{r} -\frac{1}{2}) b_2^*$ and equation of the line $\ell_1$ is given by $w = \frac{r}{k} w_1 b $. We know $b_2^* = \frac{k}{r} -\epsilon$, so 
	\begin{equation*}
	\alpha = \left(\frac{H^2}{2} +1 \right) \left(\frac{k}{r} -\epsilon\right)^2- H^2\left(\frac{k}{r} -\frac{1}{2}\right)\left(\frac{k}{r} -\epsilon\right).
	\end{equation*}
	Thus when $H^2 \rightarrow +\infty$, we know 
	\begin{equation*}
	\alpha \ \rightarrow\ \frac{H^2}{2} \left(\frac{k}{r} -\epsilon \right)\left(-\frac{k}{r} -\epsilon + 1  \right)
	\end{equation*}
	and so 
	\begin{equation}\label{limit-2}
	\frac{r}{k}w_1 \ \rightarrow\ H^2\left(\frac{k}{r} -\frac{1}{2}\right)+ \frac{H^2r}{2k} \left(\frac{k}{r} -\epsilon \right)\left(-\frac{k}{r} -\epsilon + 1  \right).
	\end{equation}
	The line $\ell_1$ intersects the vertical line $b = \frac{k}{r} -\frac{1}{r(r-1)}$ at a point with $w = \frac{r}{k} w_1(\frac{k}{r} -\frac{1}{r(r-1)})$. We claim if $g \gg 0$, this point lies in $U$, i.e.
	\begin{equation}\label{cond.e1}
	\frac{r}{k} w_1 > \Gamma\left(\frac{k}{r} -\frac{1}{r(r-1)}\right) \frac{1}{\frac{k}{r} -\frac{1}{r(r-1)}}\,.
	\end{equation}
	The limit of the right hand side when $H^2 \rightarrow +\infty$ is $\frac{H^2}{2}(\frac{k}{r} -\frac{1}{r(r-1)})$, so \eqref{limit-2} shows that \eqref{cond.e1} holds if 
	\begin{equation*}
\left(\frac{k}{r} -\epsilon \right)\left(-\frac{k}{r} -\epsilon + 1  \right) > \frac{k}{r} \left(1-\frac{k}{r} -\frac{1}{r(r-1)}\right).
	\end{equation*} 
	This is equivalent to 
	\begin{equation*}
	\epsilon -\epsilon^2 < \frac{k}{r^2(r-1)} 
	\end{equation*}
	which is satisfied by our assumption on $\epsilon = \frac{k}{r} - b_2^*$ in Definition \ref{def-ellstar}.

	\bigskip
	
	Similarly, we have
	\begin{equation*}
	w_2 = H^2\frac{k-r}{r} \left(\frac{k}{r} -\frac{1}{2} \right) + \alpha-1
	\end{equation*}
	We require 
	\begin{equation}\label{cond.e2}
	\frac{r}{k-r} w_2 < \Gamma\left(\frac{k-r}{r} +\frac{1}{r(r-1)}\right) \frac{1}{\frac{k-r}{r} +\frac{1}{r(r-1)}}\,.
	\end{equation}
	The limit of the left hand side is
	\begin{align*}
	\lim\limits_{H^2 \rightarrow +\infty} \frac{r}{k-r}w_2 =\ & H^2\left(\frac{k}{r} -\frac{1}{2}\right)+ \frac{H^2r}{2(k-r)} \left(\frac{k}{r} -\epsilon \right)\left(-\frac{k}{r} -\epsilon + 1  \right)
	\end{align*}
	and the limit of the right hand side is $\frac{H^2}{2}(\frac{k-r}{r} +\frac{1}{r(r-1)})$. Thus \eqref{cond.e2} holds for $g \gg 0$ if 
	\begin{align*}
	\frac{r}{k-r} \left(\frac{k}{r} -\epsilon \right)\left(-\frac{k}{r} -\epsilon + 1  \right) < \frac{1}{r(r-1)} -\frac{k}{r}. 
	\end{align*}
	This is equivalent to 
	\begin{equation*}
	\epsilon - \epsilon^2 < \frac{r-k}{r^2(r-1)}
	\end{equation*}
	which holds again by our assumptions on $\epsilon = \frac{k}{r} -b_2^*$. 
\end{proof}


\begin{Rem}\label{Rem-big enough}
	The genus $g = \frac{H^2}{2} +1$ is large enough for Theorem \ref{thm} if $H^2 > 2r(r+1)$ and inequalities \eqref{cond--1}, \eqref{cond--2}, \eqref{epsilon}, \eqref{cond.1}, \eqref{cond.3-before}, \eqref{cond.3}, \eqref{cond.5-before}, \eqref{cond.5}, \eqref{condition.extra}, \eqref{cond.e1}, \eqref{cond.e2}, and for the chosen values of $\epsilon, \epsilon'$ in \eqref{chosen epsilon}, inequalities \eqref{cond.2} and \eqref{cond.4} are satisfied. 
\end{Rem}

\bibliography{mybib}

\begin{thebibliography}{Fey20b}

\bibitem[ABS14]{arbarello:maukai-program}
E.~Arbarello, A.~Bruno, and E.~Sernesi.
\newblock Mukai's program for curves on a {K}3 surface.
\newblock {\em Algebr. Geom.}, 1(5):532--557, 2014.

\bibitem[Bri07]{bridgeland:stability-condition-on-triangulated-category}
T.~Bridgeland.
\newblock Stability conditions on triangulated categories.
\newblock {\em Ann. of Math. (2)}, 166(2):317--345, 2007.

\bibitem[Bri08]{bridgeland:K3-surfaces}
T.~Bridgeland.
\newblock Stability conditions on {$K3$} surfaces.
\newblock {\em Duke Math. J.}, 141(2):241--291, 2008.

\bibitem[BS13]{bhosle:brill-noether-loci-on-nodal-curves}
U.~N. Bhosle and S.~K. Singh.
\newblock Brill-{N}oether loci and generated torsion free sheaves over nodal
  and cuspidal curves.
\newblock {\em Manuscripta Math.}, 141(1-2):241--271, 2013.

\bibitem[Fey20a]{feyz:mukai-program}
S.~Feyzbakhsh.
\newblock Mukai's program (reconstructing a {K}3 surface from a curve) via
  wall-crossing.
\newblock {\em J. Reine Angew. Math.}, 765:101--137, 2020.

\bibitem[Fey20b]{feyz:mukai-program-ii}
S.~Feyzbakhsh.
\newblock Mukai's program (reconstructing a {$K3$} surface from a curve) via
  wall-crossing, {II}, arXiv:2006.08410, 2020.

\bibitem[Fey21]{feyz:effective-restriction-theorem}
S.~Feyzbakhsh.
\newblock An effective restriction theorem via wall-crossing and {M}ercat's
  conjecture, arXiv:1608.07825v2, 2021.

\bibitem[Fey22]{feyz-genus-12}
S.~Feyzbakhsh.
\newblock Curves on {K}3 surfaces and {F}ano 3-folds of genus 12, in
  preparation, 2022.

\bibitem[FL21]{feyz-li:clifford-indices}
S.~Feyzbakhsh and C.~Li.
\newblock Higher rank {C}lifford indices of curves on a {K}3 surface.
\newblock {\em Selecta Math. (N.S.)}, 27(3):Paper No. 48, 34, 2021.

\bibitem[FT21]{feyz:thomas-noether-loci}
S.~Feyzbakhsh and R.~P. Thomas.
\newblock An application of wall-crossing to {N}oether-{L}efschetz loci.
\newblock {\em Q. J. Math.}, 72(1-2):51--70, 2021.
\newblock with an appendix by C. Voisin.

\bibitem[HL10]{huybrechts:geometry-of-moduli-space-of-sheaves}
D.~Huybrechts and M.~Lehn.
\newblock {\em The geometry of moduli spaces of sheaves}.
\newblock Cambridge Mathematical Library. Cambridge University Press,
  Cambridge, second edition, 2010.

\bibitem[Huy16]{huybretchts:lectures-on-k3-surfaces}
D.~Huybrechts.
\newblock {\em Lectures on {K}3 surfaces}, volume 158 of {\em Cambridge Studies
  in Advanced Mathematics}.
\newblock Cambridge University Press, Cambridge, 2016.

\bibitem[MS17]{macri:intro-bridgeland-stability}
E.~Macr\`\i and B.~Schmidt.
\newblock Lectures on {B}ridgeland stability.
\newblock In {\em Moduli of curves}, volume~21 of {\em Lect. Notes Unione Mat.
  Ital.}, pages 139--211. Springer, Cham, 2017.

\bibitem[Muk88]{mukai:curve-k3surface-genus-less-than-10}
S.~Mukai.
\newblock Curves, {$K3$} surfaces and {F}ano {$3$}-folds of genus {$\leq 10$}.
\newblock In {\em Algebraic geometry and commutative algebra, {V}ol.\ {I}},
  pages 357--377. Kinokuniya, Tokyo, 1988.

\bibitem[Muk95]{mukai:vector-bundles-and-brill-noether}
S.~Mukai.
\newblock Vector bundles and {B}rill-{N}oether theory.
\newblock In {\em Current topics in complex algebraic geometry ({B}erkeley,
  {CA}, 1992/93)}, volume~28 of {\em Math. Sci. Res. Inst. Publ.}, pages
  145--158. Cambridge Univ. Press, Cambridge, 1995.

\bibitem[Muk96]{mukai-curves-k3surface-genus-11}
S.~Mukai.
\newblock Curves and {$K3$} surfaces of genus eleven.
\newblock In {\em Moduli of vector bundles ({S}anda, 1994; {K}yoto, 1994)},
  volume 179 of {\em Lecture Notes in Pure and Appl. Math.}, pages 189--197.
  Dekker, New York, 1996.

\bibitem[Muk01]{mukai:non-abelian-brill-noether}
S.~Mukai.
\newblock Non-abelian {B}rill-{N}oether theory and {F}ano 3-folds [translation
  of {S}\=ugaku {\bf 49} (1997), no. 1, 1--24; {MR}1478148 (99b:14012)].
\newblock {\em Sugaku Expositions}, 14(2):125--153, 2001.
\newblock Sugaku Expositions.

\bibitem[Muk02]{mukai:new-development-in-the-theory-of-fano-threefold}
S.~Mukai.
\newblock New developments in the theory of {F}ano threefolds: vector bundle
  method and moduli problems [translation of {S}\=ugaku {\bf 47} (1995), no.\
  2, 125--144; {MR}1364825 (96m:14059)].
\newblock {\em Sugaku Expositions}, 15(2):125--150, 2002.
\newblock Sugaku expositions.

\end{thebibliography}
\bibliographystyle{halpha}

\bigskip \noindent
{\tt{s.feyzbakhsh@imperial.ac.uk}}\medskip

\noindent Department of Mathematics, Imperial College, London SW7 2AZ, United Kingdom

\end{document}